\documentclass[12pt]{amsart}
\usepackage{amscd}
\usepackage{amssymb}
\usepackage{graphicx}
\usepackage{color}
\usepackage[centering,text={15.5cm,22cm},
		marginparwidth=20mm]{geometry}
\usepackage{hyperref}

\usepackage{mathrsfs}

\usepackage[all]{xy}

\newtheorem{theorem}{Theorem}[section]

\newtheorem{lemma}[theorem]{Lemma}

\newtheorem{proposition}[theorem]{Proposition}

\theoremstyle{definition}
\newtheorem{definition}[theorem]{Definition}
\newtheorem{definition-lemma}[theorem]{Definition-Lemma}

\newtheorem{example}[theorem]{Example}

\theoremstyle{remark}
\newtheorem{remark}[theorem]{Remark}

\numberwithin{equation}{section}
\numberwithin{figure}{section}

\newcommand{\bbfamily}{\fontencoding{U}\fontfamily{bbold}\selectfont}
\newcommand{\textbb}[1]{{\bbfamily#1}}

\newcommand {\lfor} {\mbox{\textbb{[}}}
\newcommand {\rfor} {\mbox{\textbb{]}}}

\newcommand{\ZZ} {\mathbb{Z}}
\newcommand{\QQ} {\mathbb{Q}}
\newcommand{\RR} {\mathbb{R}}
\newcommand{\bR} {\RR}

\newcommand{\CC} {\mathbb{C}}

\newcommand{\PP} {\mathbb{P}}
\newcommand{\VV} {\mathbb{V}}
\renewcommand{\AA} {\mathbb{A}}
\newcommand{\GG} {\mathbb{G}}

\newcommand {\shM}  {\mathcal{M}}

\newcommand {\shO}  {\mathcal{O}}

\newcommand {\foD}  {\mathfrak{D}}

\newcommand {\foX}  {\mathfrak{X}}

\newcommand {\fod}  {\mathfrak{d}}

\newcommand {\fom}  {\mathfrak{m}}

\newcommand {\can}  {\mathrm{can}}

\newcommand {\GL}  {\operatorname{GL}}

\newcommand {\gp}  {{\operatorname{gp}}}

\newcommand {\id}  {\operatorname{id}}

\newcommand {\Int}  {\operatorname{Int}}

\newcommand {\kk} {\Bbbk}

\newcommand {\Mono} {\operatorname{Mono}}

\newcommand {\out}  {\mathrm{out}}

\newcommand {\PGL}  {\operatorname{PGL}}

\newcommand {\SL}  {\operatorname{SL}}
\newcommand {\Spec} {\operatorname{Spec}}

\newcommand {\Spf}  {\operatorname{Spf}}

\newcommand {\NE}   {\operatorname{NE}}

\def\bP{\Bbb P}

\def\SL{\operatorname{SL}}

\def\GL{\operatorname{GL}}

\def\bZ{\Bbb Z}

\def\bR{\Bbb R}

\def\oD{\bar{D}}

\def\tS{\tilde S}

\def\oS{\overline{S}}

\def\tS{\tilde{S}}

\def\oN{\overline{N}}

\def\Spec{\operatorname{Spec}}

\def\PGL{\operatorname{PGL}}

\def\Spec{\operatorname{Spec}}

\def\Efi#1#2#3#4#5{\displaystyle
#1\!\!-\!\!#2
\!\!-\!\!#3
\!\!-\!\!#4
\hskip-24.2pt\lower4.5pt\hbox{${\scriptstyle|}
\hskip-3.35pt\lower6pt\hbox{$#5$}$}}

\def\Evia#1#2#3#4#5{\displaystyle
#1\!\!-\!\!#2
\!\!-\!\!#3
\hskip-24.2pt\lower4.5pt\hbox{${\scriptstyle|}
\hskip-3.35pt\lower6pt\hbox{$#4\!\!-\!\!\!-\!\!\!-\!\!$}$\hskip2.3pt${\scriptstyle|}
\hskip-3.35pt\lower6pt\hbox{$#5$}$}}

\def\Ezia#1#2#3#4{\displaystyle
#1\!\!-\!\!#2
\hskip-14.8pt\lower4.5pt\hbox{${\scriptstyle|}
\hskip-3.35pt\lower6pt\hbox{$#3\!\!-\!\!$}$\hskip2.3pt${\scriptstyle|}
\hskip-3.35pt\lower6pt\hbox{$#4$}$}}

\def\Efia#1#2#3#4#5#6{\displaystyle
#1\!\!-\!\!#2
\!\!-\!\!#3
\!\!-\!\!#4
\hskip-24.2pt\lower4.5pt\hbox{${\scriptstyle|}
\hskip-3.35pt\lower6pt\hbox{$#5$}$\hskip5.7pt${\scriptstyle|}
\hskip-3.35pt\lower6pt\hbox{$#6$}$}}

\def\Esi#1#2#3#4#5#6{\displaystyle
#1\!\!-\!\!#2
\!\!-\!\!#3
\!\!-\!\!#4\!\!-\!\!#5
\hskip-24.2pt\lower4.5pt\hbox{${\scriptstyle|}
\hskip-3.35pt\lower6pt\hbox{$#6$
\lower3pt\hbox{\ }}$}}

\def\Esia#1#2#3#4#5#6#7{\displaystyle

#1\!\!-\!\!#2
\!\!-\!\!#3
\!\!-\!\!#4\!\!-\!\!#5
\hskip-24.2pt\lower4.5pt\hbox{${\scriptstyle|}
\hskip-3.35pt\lower6pt\hbox{$#6$\hskip-3.8pt\lower4.5pt\hbox{${\scriptstyle|}
\hskip-3.35pt\lower6pt\hbox{$#7$}$}}
\lower3pt\hbox{\ }$}}

\def\Ese#1#2#3#4#5#6#7{\displaystyle
#1\!\!-\!\!#2
\!\!-\!\!#3
\!\!-\!\!#4\!\!-\!\!#5\!\!-\!\!#6
\hskip-33.6pt\lower4.5pt\hbox{${\scriptstyle|}
\hskip-3.35pt\lower6pt\hbox{$#7$
\lower3pt\hbox{\ }
}$}}

\def\Esea#1#2#3#4#5#6#7#8{\displaystyle
#1\!\!-\!\!#2
\!\!-\!\!#3
\!\!-\!\!#4\!\!-\!\!#5\!\!-\!\!#6\!\!-\!\!#7
\hskip-33.6pt\lower4.5pt\hbox{${\scriptstyle|}
\hskip-3.35pt\lower6pt\hbox{$#8$
\lower3pt\hbox{\ }
}$}}

\def\Eei#1#2#3#4#5#6#7#8{\displaystyle
#1\!\!-\!\!#2
\!\!-\!\!#3
\!\!-\!\!#4\!\!-\!\!#5\!\!-\!\!#6\!\!-\!\!#7
\hskip-43.2pt\lower4.5pt\hbox{${\scriptstyle|}
\hskip-3.35pt\lower6pt\hbox{$#8$
\lower3pt\hbox{\ }
}$}}

\def\Eeia#1#2#3#4#5#6#7#8#9{{\displaystyle
#1\!\!-\!\!#2
\!\!-\!\!#3
\!\!-\!\!#4\!\!-\!\!#5\!\!-\!\!#6\!\!-\!\!#7\!\!-\!\!#8
\hskip-52.2pt\lower4.5pt\hbox{${\scriptstyle|}
\hskip-3.35pt\lower6pt\hbox{$#9$
\lower3pt\hbox{\ }
}$}}}

\def\os{\overline{S}}

\def\tS{\tilde{S}}
\def\ts{\tS}
\def\ts7{\tilde{S}_7}

\def\bP{\Bbb P}

\def\bZ{\Bbb Z}

\def\bR{\Bbb R}

\def\tS{\tilde S}

\def\oS{\overline{S}}

\def\tS{\tilde{S}}

\def\oN{\overline{N}}

\def\Spec{\operatorname{Spec}}

\def\PGL{\operatorname{PGL}}

\def\oN{\overline{N}}
\def\os7p{\oS_7'}
\def\os7{\oS_7}
\def\on6{\oN_6}
\def\n6{\oN_6}

\def\Mono{\operatorname{Mono}}

\def\brg0{(\bR_{\geq 0})}

\def\mydate{\ifcase\month \or January\or February\or March\or
April\or May\or June\or July\or August\or September\or October\or 
November\or December\fi \space\number\day,\space\number\year}

\begin{document}

\title[The mirror of the cubic surface]
{The mirror of the cubic surface}
\dedicatory{To Miles Reid on the occasion of his $70^{th}$ birthday}
\author{Mark Gross} 
\address{DPMMS, University of Cambridge,
Wilberforce Road, Cambridge CB3 0WB, UK}
\email{mgross@dpmms.cam.ac.uk}
\author{Paul Hacking}
\address{Department of Mathematics and Statistics, Lederle Graduate
Research Tower, University of Massachusetts, Amherst, MA 01003-9305}
\email{hacking@math.umass.edu}

\author{Sean Keel}
\author{Bernd Siebert}
\address{Department of Mathematics, 1 University Station C1200, Austin,
TX 78712-0257}
\email{keel@math.utexas.edu, siebert@math.utexas.edu}
\maketitle
\tableofcontents
\bigskip

\section*{Introduction}

A number of years ago, one of us (M.G.) was giving a lecture at the
University of Warwick on the material on scattering diagrams from 
\cite{GPS}. Of course Miles was
in the audience, and he asked (paraphrasing as this was many years ago)
whether, at some point, the lecturer would come back down to earth.
The goal of this note is to show in fact we have not left the planet
by considering a particularly beautiful example of the mirror symmetry
construction of \cite{GHK11}, namely the mirror to a cubic surface.

More precisely, the paper \cite{GHK11}, building on \cite{GS11}, \cite{GPS} 
and \cite{CPS}, constructs mirrors of rational surfaces equipped with
anti-canonical cycles of rational curves. Specifically, one begins
with the data of a pair $(Y,D)$,
where $Y$ is a non-singular projective rational surface over an algebraically
closed field $\kk$ of characteristic $0$, and $D\in |-K_Y|$ is an effective
reduced anti-canonical divisor with at least one node, necessarily then
forming a wheel of projective lines. Choose in addition
a finitely generated, saturated sub-monoid $P\subset H_2(Y,\ZZ)$ 
whose only invertible element is $0$, such that $P$ contains the class
of every effective curve on $Y$. Let $\fom$ denote the maximal
monomial ideal of the monoid ring $\kk[P]$ and $\widehat{\kk[P]}$ denote
the completion of $\kk[P]$ with respect to $\fom$. Then the main
construction of \cite{GHK11} produces a family of formal schemes
$\foX\rightarrow \Spf \widehat{\kk[P]}$ which is interpreted as
the mirror family to the pair $(Y,D)$. In the more pleasant case when $D$
supports an ample divisor, the construction is in fact algebraic:
there is a family $X\rightarrow S:=\Spec \kk[P]$ of affine surfaces extending
the above formal family. In general, if $D$ has $n\ge 3$ components,
then $X$ is a closed subscheme of $\AA^n_S$, with central fibre a reducible
union of $n$ copies of $\AA^2$.

\cite{GHK11}, Example 6.13, contains the 
equation\footnote{Unfortunately with a sign error!} for $X$
in the case that $Y$ is a cubic surface in $\PP^3$ and $D=D_1+D_2+D_3$ 
is a triangle of lines. The intent was to include a proof of
this in \cite{GHKII},
which, at the time, was circulated rather narrowly in an extreme rough
draft form. 

As \cite{GHKII} has seen no change for more than five years, and many
pieces of it have been cannibalized for other papers or become out-of-date,
it seemed that, in the grand tradition of second parts of
papers, this paper is unlikely to ever see the light of day.
On the other hand, the full details of the cubic surface have not appeared 
anywhere else, although Lawrence Barrott \cite{B18} verifies the 
given equation for the mirror of the cubic surface.
The cubic surface is in particular especially attractive.
This is unsurprising, given the rich classical geometry of the cubic
(see e.g., \cite{Reid}).
So we felt that it would be a pity for this construction 
never to appear. Further, since \cite{GHKII} first began to circulate, 
the technology for understanding the product rule for theta functions
on the mirror, and hence the equations for the mirror, has improved
at a theoretical level, see \cite{GS18},\cite{GS19},\cite{KY19}.
Thus in particular it will be possible to give
a completely enumerative interpretation for the equations to the mirror
cubic. This gives us an opportunity to exposit a number of different
viewpoints on the construction here.

Without further ado, here is the main result. Describe the pair
$(Y,D)$ as follows. First fix the pair $(\PP^2, \oD=\oD_1+\oD_2+\oD_3)$
where $\oD$ is a triangle of lines. Let $(Y,D)$ be obtained as
the blow-up of two general distinct points on each of the three lines,
with $D$ the strict transform of $\oD$. Let $E_{ij}$, $i=1,2,3$, $j=1,2$
be the exceptional curves, with $E_{ij}$ intersecting $D_i$. For $i=1,2$ or
$3$, denote by $L_{ij}$, $1\le j \le 8$, the eight lines on
the cubic surface not contained in $D$ but intersecting $D_i$. We note
that $\{E_{i1},E_{i2}\}\subseteq \{L_{ij}\,|\,1\le j \le 8\}$.

\begin{theorem}
\label{thm:main theorem}
Taking $P=\NE(Y)$, the cone of effective curves of $Y$, $S=\Spec \kk[P]$,
the mirror family defined over $S$ to the cubic surface $(Y,D=D_1+D_2+D_3)$ 
is given by the
equation in $\AA^3_{S}$:
\[
\vartheta_1\vartheta_2\vartheta_3= \sum_i z^{D_i} \vartheta_i^2 +
\sum_i \left(\sum_j z^{L_{ij}}\right)z^{D_i} \vartheta_i +
\sum_{\pi} z^{\pi^*H} + 4z^{D_1+D_2+D_3}.
\]
Here for a curve class $C$, $z^C$ denotes the corresponding monomial
of $\kk[P]$, and $\vartheta_1,\vartheta_2,\vartheta_3$ are the coordinates
on the affine $3$-space.
The sum over
$\pi$ is the sum over all possible birational morphisms 
$\pi \colon Y \rightarrow Y'$
of $(Y,D)$ to a pair $(Y',D')$  isomorphic to $\bP^2$ with its toric boundary,
with $\pi|_{D}:D\rightarrow D'$ an isomorphism and $H$ the class of a line
in $\PP^2$.
\end{theorem}

The original guess for the shape of these equations was motivated by
the paper \cite{Ob04}, which gave a similar equation for a non-commutative
cubic surface. Once one knows the shape of the equation, it is not
difficult to verify it, as we shall see.

Finally, we note that this paper does not intend to be 
a complete exposition of the ideas of \cite{GHK11}, but rather, we
move quickly to discuss the cubic surface. For a more comprehensive
expository account, see the forthcoming work of Arg\"uz \cite{Ar}.

\bigskip

\emph{Acknowledgements:} We would like to thank L.\ Barrott, 
A.\ Neitzke, A.\ 
Oblomkov and Y.\ Zhang for useful discussions. M.G.\ was supported by
EPSRC grant EP/N03189X/1 and a Royal Society Wolfson Research
Merit Award. P.H.\ was supported by NSF grant DMS-1601065 and
DMS-1901970. S.K.\ was
supported by NSF grant DMS-1561632.

\section{The tropicalization of the cubic surface}

We explain the basic combinatorial data we associate to the pair
$(Y,D)$, namely a pair $(B,\Sigma)$ where:
\begin{itemize}
\item $B$ is an \emph{integral linear manifold with singularities};
\item $\Sigma$ is a decomposition of $B$ into cones.
\end{itemize}
First, an integral linear manifold $B$ is a real manifold with coordinate
charts $\psi_i:U_i\rightarrow\RR^n$ (where $\{U_i\}$ is an open covering
of $B$) and transition maps $\psi_i\circ \psi_j^{-1}\in
\GL_n(\ZZ)$. An \emph{integral linear manifold with singularities}
is a manifold $B$ with an open set $B_0\subseteq B$ and 
$\Delta=B\setminus B_0$ of codimension at least $2$ such that $B_0$
carries an integral linear structure.

We build $B$ and $\Sigma$ by pretending that the pair $(Y,D)$ is
a toric variety. If it were, we could reconstruct its fan in
$\RR^2$ (up to $\GL_2(\ZZ)$) knowing the intersection numbers of
the irreducible components $D_i$ of $D$. So we just start
constructing a fan and we will run into trouble when $(Y,D)$ isn't a
toric variety. This problem is fixed by introducing a singularity in the linear
structure of $\RR^2$ at the origin.

Explicitly, for the cubic surface, $D_i^2=-1$ for $1\le i \le 3$, and we proceed
as follows. Take rays in $\RR^2$ corresponding to $D_1$ and $D_2$ to
be $\rho_1:=\RR_{\ge 0} (1,0)$ and $\rho_2:=\RR_{\ge 0}(0,1)$ respectively.
See the left-hand picture in Figure \ref{Figure1}.

\begin{figure}
\input{Figure1.pstex_t}
\caption{}
\label{Figure1}
\end{figure}

Since $D_2^2=-1$, toric geometry instructs us that the ray corresponding to
$D_3$ would be $\rho_3:=\RR_{\ge 0}(-1,1)$ if $(Y,D)$ were a toric pair.
Indeed, if $\rho_1,\rho_2,\rho_3$ are successive rays in a two-dimensional
fan defining a non-singular complete toric surface, and if $n_i$ is
the primitive generator of $\rho_i$ and $D_i$ is the divisor
corresponding to $\rho_i$, we have the relation
\[
n_1+D_2^2 n_2 + n_3=0.
\]
Thus with $n_1=(1,0)$ and $n_2=(0,1)$, $n_3$ is determined by 
$D_2^2$. As
$D_3^2=-1$, we then need a ray corresponding to $D_1$ to be $\RR_{\ge 0}(-1,0)$,
which does not coincide with the ray $\rho_1$ (telling us that $(Y,D)$
wasn't really a toric pair). If we continue, we obtain a new ray
$\RR_{\ge 0}(0,-1)$ for $D_2$ also. Thus we have two cones spanned by the rays
corresponding to $D_1$ and $D_2$, and there is an integral linear transformation
identifying these two cones. In this case, this transformation
is $-\id$. After cutting out the fourth quadrant from $\RR^2$
and gluing the third and first quadrants via $-\id$, we obtain
the integral affine manifold $B$, along with a decomposition (or fan)
$\Sigma$ into rational polyhedral cones. Here the cones of $\Sigma$
consist of $\{0\}$, the images of the rays $\rho_1,\rho_2,\rho_3$,
and three two-dimensional cones $\sigma_{i,i+1}$ $i=1,2,3$, with
indices taken mod $3$, and $\sigma_{i,i+1}$ having faces $\rho_i$
and $\rho_{i+1}$. Note the rays correspond to irreducible components of $D$
and the two-dimensional cones to double points of $D$.

To see the details of this construction in general, and further
examples, see \cite{GHK11}, \S1.2.

We use the convention that $v_i\in B$ is the primitive integral point
on the ray $\rho_i$, so that any element of $\sigma_{i,i+1}$ can
be written as $av_i+bv_{i+1}$ for some $a,b\in \RR_{\ge 0}$.

While we have just described the general construction for $(B,\Sigma)$
as applied to our particular case, in fact there is a more
elegant description for the cubic surface. By continuing to build the
fan, we close up to get a fan $\widetilde\Sigma$ in $\RR^2$.
Then $B=\RR^2/\langle -\id\rangle$,
and $\widetilde\Sigma$ descends to $\Sigma$ on the quotient.
See the right-hand side of Figure \ref{Figure1}.

\begin{remark} 
\label{rem:cayley cubic}
Note that $\widetilde\Sigma$ defines a toric variety $\widetilde Y$
which is a del Pezzo surface of degree $6$. 
The automorphism $-\id$ of $\widetilde\Sigma$ induces an involution 
$\iota:\widetilde Y\rightarrow\widetilde Y$, which is given on the dense
torus orbit as $(z_1,z_2)\mapsto (z_1^{-1},z_2^{-1})$. This surface can
be embedded in $\PP^3$ as follows. First, one maps the quotient
of the dense torus orbit of $\widetilde Y$ to $\AA^3$ using the map
\[
(z_1,z_2)\mapsto (z_1+z_1^{-1}, z_2+z_2^{-1},z_1^{-1}z_2+z_1z_2^{-1}).
\]
The image satisfies the equation $x_1x_2x_3=x_1^2+x_2^2+x_3^2-4$, which is
then projectivized to obtain the Cayley cubic given by the equation
\[
x_1x_2x_3=x_0(x_1^2+x_2^2+x_3^2)-4x_0^3,
\]
the unique cubic surface with four ordinary double points, the images
of the fixed points of $\iota$. This is in fact isomorphic to 
$\widetilde Y/\langle\iota\rangle$.
\qed
\end{remark}

\begin{remark}
\label{rem:trop}
We write $B_0=B\setminus\{0\}$. Let $B_0(\ZZ)$ be the subset
of $B_0$ of points with integer coordinates with respect to any integral
linear chart. Set $B(\ZZ)=B_0(\ZZ)\cup \{0\}$.

The set $B(\ZZ)$ has another natural interpretation, as
the \emph{tropicalization} of the log Calabi-Yau manifold $U=Y\setminus D$,
see \cite{GHK13}, Definition 1.7. Here one takes a nowhere vanishing
$2$-form $\Omega$ on $U$ with at worst simple poles along $D$, and we have
\[
B(\ZZ)=\{\hbox{divsorial discrete valuations $\nu:k(U)^*\rightarrow\ZZ$}
\,|\,\nu(\Omega)<0\}\cup \{0\}.
\]
The advantage of this description is that an automorphism of $U$ which does
not extend to an automorphism of $Y$ still induces an automorphism of
$B(\ZZ)$, which in general extends to a piecewise linear automorphism
of $B$. \qed
\end{remark}

It will also be useful to consider piecewise linear functions on $B$
with respect to the fan $\Sigma$, i.e., continuous functions $F:B\rightarrow
\RR$ which restrict to linear functions on each $\sigma\in\Sigma$.
Just as in the toric case, there is in fact a one-to-one correspondence 
between such functions with integral slopes and divisors supported on
the boundary. Indeed, each boundary divisor $D_i$
defines a piecewise linear function on $B$, written as $\langle D_i,\cdot\rangle$,
uniquely defined by the requirement that
\[
\langle D_i,v_j\rangle=\delta_{ij}.
\]
Additively, this allows us to obtain a PL function $\langle D',\cdot\rangle$
associated to any divisor $D'$ supported on $D$. Conversely, given
a piecewise linear function $F:B\rightarrow\RR$ with integral slopes,
we obtain a divisor $\sum_i F(v_i)D_i$ supported on $D$.

\section{The scattering diagram associated to the cubic surface}

As $(B,\Sigma)$ only involves purely combinatorial information about
$(Y,D)$, it is insufficient to determine an interesting mirror
object. We need to include extra data of a \emph{scattering diagram}
on $(B,\Sigma)$. 

Before doing so, we need to select some additional auxiliary data,
namely a monoid $P\subseteq H_2(Y,\ZZ)$ of the form
$\sigma_P\cap H_2(Y,\ZZ)$ for $\sigma_P\subseteq H_2(Y,\RR)$
a strictly convex rational polyhedral cone which contains all effective
curve classes. In the case of the cubic surface, we may take
$\sigma_P$ to be the Mori cone, i.e., the cone generated by all
effective curve classes. We write $\kk[P]$ for the corresponding monoid
ring, and $\fom\subseteq\kk[P]$ for the maximal monomial ideal,
generated by $\{z^p\,|\, p\in P\setminus \{0\}\}$. 

\begin{definition}
\label{def:ray}
A \emph{ray} in $B$ is a pair $(\fod,f_{\fod})$ where:
\begin{enumerate}
\item $\fod\subseteq \sigma_{i,i+1}$ for some $i$  
is a ray generated by some $av_i+bv_{i+1}\not=0$, $a,b\in \ZZ_{\ge 0}$
relatively prime. We call $\fod$ the \emph{support} of the ray.
\item $f_{\fod}=1+\sum_{k\ge 1} c_k X_i^{-ak}X_{i+1}^{-bk}
\in \kk[P]\lfor X_i^{-a}X_{i+1}^{-b}\rfor$ with $c_k\in\fom$ for all
$k$, satisfying the property
that for any monomial ideal $I\subseteq\kk[P]$ with $\kk[P]/I$
Artinian, (i.e., $I$ is co-Artinian), $f_{\fod}\mod I$ is a finite sum.
\end{enumerate}
\end{definition}

\begin{definition}
A \emph{scattering diagram} $\foD$ for $B$ is a collection of rays
with the property that for each co-Artinian monomial ideal $I\subseteq\kk[P]$,
\[
\foD_I:=\{(\fod,f_{\fod})\in \foD\,|\, f_{\fod}\not\equiv 1\mod I\}
\]
is finite. We assume further that $\foD$ contains at most one ray with
a given support.
\end{definition}

The purpose of a scattering diagram is to give a way of building a flat
family over $\Spec A_I$, where $A_I:=\kk[P]/I$. Explicitly, suppose given a 
scattering diagram $\foD$ and a co-Artinian ideal $I$. Assume
further that each $\rho_i$ is the support of a ray $(\rho_i,f_i)\in
\foD$. This ray is allowed to be trivial, i.e., $f_i=1$. Now define
rings
\begin{align*}
R_{i,I}:= {} & A_I[X_{i-1},X_i^{\pm 1},X_{i+1}]/(X_{i-1}X_{i+1}
-z^{[D_i]}X_i^{-D_i^2}f_i)\\
R_{i,i+1,I}:={}  & A_I[X_i^{\pm 1}, X_{i+1}^{\pm 1}].
\end{align*}
Here $z^{[D_i]}$ is the monomial in $\kk[P]$ corresponding to the class
of the boundary curve $D_i$, necessarily lying in $P$ by the assumption that
$P$ contains all effective curve classes. 

Localizing, note we have canonical isomorphisms 
\[
\hbox{$(R_{i,I})_{X_{i+1}}\cong R_{i,i+1,I}$ and
$(R_{i,I})_{X_{i-1}}\cong R_{i-1,i,I}$}.
\]

Set
\[
\hbox{$U_{i,I}:=\Spec R_{i,I}$ and $U_{i,i+1,I}:=\Spec R_{i,i+1,I}$.}
\]
Note that if $I=\fom$, then $U_{i,I}$ is the reducible
variety defined by $X_{i-1}X_{i+1}=0$ in $\AA^2_{X_{i-1},X_{i+1}}
\times(\GG_m)_{X_i}$, where the subscripts denote the coordinates
on the respective factors. On the other hand, $U_{i,i+1,I}=(\GG_m^2)_{X_i,
X_{i+1}}$. For more general $I$, we instead obtain thickenings
of these schemes just described.

For any $I$, we have canonical open immersions $U_{i-1,i,I},U_{i,i+1,I}
\hookrightarrow U_{i,I}$. As $U_{i,I}$ and $U_{i,\fom}$ have the same
underlying topological space, we can describe the underlying
open sets in $U_{i,\fom}$ of these two open immersions as subsets of
$V(X_{i-1}X_{i+1}) \subseteq \AA^2\times \GG_m$ as follows. We have
$U_{i-1,i,I}$ is given by the open set where
$X_{i-1}\not=0$ (hence $X_{i+1}=0$) and 
$U_{i,i+1,I}$ is given by the open set where
$X_{i+1}\not=0$ (hence $X_{i-1}=0$).
Thus in particular the images of these immersions are
disjoint. Thus, if for all $i$ we glue $U_{i,I}$ and $U_{i+1,I}$
via the canonically identified copies of $U_{i,i+1,I}$, there
is no cocycle gluing condition to check
and we obtain a scheme $X_I^{\circ}$ flat over $\Spec A_I$.

It is easy to describe this if we take $I=\fom$.
One obtains
in this case that $X_I^{\circ}=\VV_n\setminus \{0\}$, where $n$ is
the number of irreducible components of $D$ and, assuming $n\ge 3$,
\[
\VV_n=\AA^2_{x_1,x_2}\cup\cdots\cup \AA^2_{x_{n-1},x_n}
\cup \AA^2_{x_n,x_1}\subseteq \AA^n=\Spec\kk[x_1,\ldots,x_n],
\]
where $\AA^2_{x_i,x_{i+1}}$ denotes the affine coordinate plane
in $\AA^n$ for which all coordinates but $x_i,x_{i+1}$ are zero.
Here, $\VV_n$ is called the \emph{$n$-vertex}.

The problem is that for $I$ general, $X^\circ_I$ may be insufficiently
well-behaved to extend to a flat deformation of $\VV_n$. To do so, we need
to perturb the
gluings we made above, and the role of the scattering diagram is to 
provide a data structure for doing so.

Let $\gamma:[0,1]\rightarrow \Int(\sigma_{i,i+1})$ be a path. We define
an automorphism of $R_{i,i+1,I}$ called the \emph{path ordered product}.
Assume that whenever $\gamma$ crosses a ray in $\foD_I$ it passes
from one side of the ray to the other. In particular, suppose $\gamma$
crosses a given ray 
\[
(\fod=\RR_{\ge 0}(av_i+bv_{i+1}), f_{\fod})\in \foD_I
\]
with $a,b$ relatively prime.
Define the $A_I$-algebra homomorphism
$\theta_{\gamma,\fod}:R_{i,i+1,I}\rightarrow R_{i,i+1,I}$
by
\begin{align*}
\theta_{\gamma,\fod}(X_i)={} & X_if_{\fod}^{\mp b}\\
\theta_{\gamma,\fod}(X_{i+1})={}& X_{i+1}f_\fod^{\pm a}
\end{align*}
where the signs are $-b, +a$ if $\gamma$ passes from the $\rho_{i+1}$
side of $\fod$ to the $\rho_i$ side of $\fod$, and $+b,-a$ if $\gamma$
crosses in the opposite direction.
Note these two choices are inverse automorphisms
of $R_{i,i+1,I}$, and $f_{\fod}$ is invertible
because $f_{\fod}\equiv 1 \mod \fom$ from Definition \ref{def:ray}, (2).

If $\gamma$ crosses precisely the rays
$(\fod_1,f_{\fod_1}),\ldots, (\fod_s,f_{\fod_s})\in \foD_I$, in that order,
then we define the \emph{path ordered product}
\[
\theta_{\gamma,\foD}:=\theta_{\gamma,\fod_s}\circ\cdots\circ
\theta_{\gamma,\fod_1}.
\]

Now, for each $i$, choose a path $\gamma$ inside $\sigma_{i,i+1}$ which starts
near $\rho_{i+1}$ and ends near $\rho_i$ so that it crosses 
all rays of $\foD_I$ intersecting the interior of $\sigma_{i,i+1}$.
Then $\theta_{\gamma,\foD}$ induces an automorphism 
$\theta_{\gamma,\foD}:U_{i,i+1,I}\rightarrow U_{i,i+1,I}$,
and we can use this to modify our gluing via
\[
\xymatrix@C=30pt
{
U_{i,I}& U_{i,i+1,I}\ar@{_{(}->}[l]\ar[r]^{\theta_{\gamma,\foD}}&U_{i,i+1,I}
\ar@{^{(}->}[r]&U_{i+1,I}.
}
\]
This produces a new scheme $X^{\circ}_{I,\foD}$, still a flat deformation
of $\VV_n\setminus \{0\}$ over $\Spec A_I$.

Now comes the key point: we need to make a good choice of $\foD$ in
order to be able to construct a partial compactification $X_{I,\foD}$
of $X^{\circ}_{I,\foD}$ such that $X_{I,\foD}\rightarrow \Spec A_I$
is a flat deformation of $\VV_n$.
One of the main ideas of \cite{GHK11} is the use of results of
\cite{GPS} to write down a good choice of scattering diagram,
the \emph{canonical scattering diagram}, in terms of
relative Gromov-Witten invariants of the pair $(Y,D)$.

We first discuss the nature of these invariants. Choose
a curve class $\beta$ and a point $v\in B_0(\ZZ)$, say
$v=av_i+bv_{i+1}$. We sketch the construction of a Gromov-Witten type
invariant $N^{\beta}_v$ counting
what we call \emph{$\AA^1$-curves}. Roughly speaking, these 
are one-pointed stable
maps of genus $0$, $f:(C,p)\rightarrow Y$, representing the
class $\beta$, with $f^{-1}(D)=\{p\}$. Further, 
$f$ has contact order $\langle D_i,v\rangle$ with $D_i$ at $p$.
Roughly, this contact order is the order of vanishing of
the regular function $f^*(t)$ at $p$, for $t$ a local defining
equation for $D_i$ at $f(p)$.
However, as stated, this isn't quite right because of standard
issues of compactness in relative Gromov-Witten theory.
In \cite{GHK11}, these numbers are defined rigorously following
\cite{GPS} by peforming a weighted
blow-up of $(Y,D)$ at $D_i\cap D_{i+1}$ determined by $\fod$ and then using
relative Gromov-Witten theory. As relative Gromov-Witten theory only
works relative to a smooth divisor, one removes all double points
of the proper transform of $D$ under this blow-up, and then shows that
this doesn't interfere with compactness of the moduli space.
We refer to \cite{GHK11}, \S3.1 for the precise definition, as we will
not need here the subtleties of the general definition.
However, we note that in order for such a map to exist, and hence
possibly have $N^\beta_v\not=0$, we must have $\beta$ an effective
curve class and 
\[
\beta\cdot D_j=\langle D_j,v\rangle.
\]

A more modern definition of these invariants is via logarithmic
Gromov-Witten theory, as developed by \cite{JAMS},\cite{AC14},\cite{Chen14}. 
Using that theory, one
can allow contact orders with multiple divisors simultaneously, and thus
do not need to perform the weighted blow-up. It follows from invariance
of logarithmic Gromov-Witten theory under toric blow-ups \cite{AW18}
and
the comparison theorem of relative and logarithmic invariants \cite{AMW} that
these two definitions agree.

\begin{definition}
The \emph{canonical scattering diagram} $\foD_{\can}$ of $(Y,D)$
consists of rays $(\fod,f_{\fod})$ ranging over all possible supports
$\fod\subseteq B$ where, if $\fod\subseteq\sigma_{i,i+1}$ with
$\fod=\RR_{\ge 0}(av_i+bv_{i+1})$ and $a,b$ relatively prime,
then
\[
f_{\fod}=\exp\left(\sum_{k\ge 1}\sum_{\beta\in H_2(Y,\ZZ)}
kN^{\beta}_{akv_i+bkv_{i+1}} z^{\beta}(X_i^{-a}X_{i+1}^{-b})^k\right).
\]
\end{definition}

\bigskip

We now return to the cubic, where $\foD_{\can}$ is particularly interesting.
One might also consider higher degree del Pezzo surfaces. However, del Pezzo
surfaces of degree $6,7,8$ and $9$ are all toric, assuming one takes
as $D$ the toric boundary, and they have a trivial scattering diagram (i.e.,
all $f_{\fod}=1$ as the invariants $N^{\beta}_v$ are always zero).
The case of a degree $5$ del Pezzo surface was considered as a running
example in \cite{GHK11}, see e.g., Example 3.7 there. A degree $4$ 
surface is not that much more complicated, see \cite{B18} for details.
On the other hand, for the cubic surface, no $f_{\fod}$
is $1$, but nevertheless we can essentially determine $f_{\fod}$.
On the other hand, the degree $2$ del Pezzo surface 
requires use of a computer to analyze, see \cite{B18}.

To describe curve classes on the cubic surface $Y$, we use
the description of $Y$ as a blow-up of $\PP^2$ given in the
introduction, so that $H_2(Y,\ZZ)$ is generated by the classes
of the exceptional divisors $E_{ij}$, $1\le i\le 3$, $1\le j\le 2$,
and the class $L$ of a pull-back of a line in $\PP^2$.

With this notation, we have:

\begin{proposition}
\label{prop:first ray}
The ray $(\rho_i,f_{\rho_i})$ satisfies
\[
f_{\rho_i}={\prod_{j=1}^8 (1+z^{L_{ij}} X_i^{-1})
\over (1-z^{D_k+D_\ell}X_i^{-2})^4},
\]
where the $L_{ij}$ as in the introduction are the lines not contained
in $D$ but meet
$D_i$, and $\{i,k,\ell\}=\{1,2,3\}$.
\end{proposition}

\begin{proof}
We take $i=1$, the other cases following from symmetry.
We need to calculate the numbers $N^{\beta}_{kv_1}$. In particular,
for $\beta$ to be represented by an $\AA^1$-curve contributing
to $N^{\beta}_{kv_1}$, we must have 
$\beta\cdot D_1=k$ and $\beta\cdot D_i=0$ for $i\not=1$.

We will first consider those curve classes $\beta$ which may be
the curve class of a generically injective map $f:\PP^1\rightarrow
Y$ with the above intersection numbers with the $D_i$. Write
\[
\beta=aL-\sum_{i,j} b_{ij}E_{ij}.
\]
Then 
\[
k=\beta\cdot D_1=a-b_{11}-b_{12}, \quad
0=\beta\cdot D_2=a-b_{21}-b_{22}, \quad
0=\beta\cdot D_3=a-b_{31}-b_{32}.
\]
Thus
\[
a=b_{21}+b_{22}=b_{31}+b_{32}.
\]

Further, we must have $p_a(f(C))\ge 0$, so by adjunction and the
fact that $K_Y=-D$,
\begin{equation}
\label{eq:adj}
-2\le 2p_a(f(C))-2=\beta\cdot(\beta+K_Y)=a^2-\sum_{i,j} b_{ij}^2-k.
\end{equation}

Now of course the curve classes $E_{11}$, $E_{12}$ satisfy the above
equalities and inequality, with $k=1$, while 
$E_{ij}$, $i\not=1$ do not.
Then any other class of an irreducible curve which may
contribute necessarily
has $a>0$ and $b_{ij}\ge 0$. Let us fix $a$ and $k$ and try to maximize
the right-hand side of \eqref{eq:adj} in the hopes that we can
make it at least $-2$. This means in particular that we should try to minimize
$b_{i1}^2+b_{i2}^2$ for $i=2,3$.

We split the analysis into two cases. If $a$ is even, then this sum
of squares is minimized by taking $b_{i1}=b_{i2}=a/2$. Thus we
see that 
\[
-2\le 2p_a(f(C))-2\le a^2-k-b_{11}^2-b_{12}^2-4(a^2/4)=-k-b_{11}^2-b_{12}^2.
\]
Since $k\ge 1$, we see we immediately get three possibilities:
\begin{enumerate}
\item $k=1$, $b_{11}=1$, $b_{12}=0$, in which case $a=2$ and 
the only possible curve class is $\beta=2L-E_{11}-E_{21}-\cdots-E_{32}$.
\item $k=1$, $b_{11}=0$, $b_{12}=1$, in which case $a=2$ and 
the only possible curve class is $\beta=2L-E_{12}-E_{21}-\cdots-E_{32}$.
\item $k=2$, $b_{11}=b_{12}=0$, in which case $a=2$ and the only 
possible curve class is $2L-E_{21}-\cdots-E_{32}$.
\end{enumerate}

If $a$ is odd, then we minimize $b_{i1}^2+b_{i2}^2$ by taking
$b_{i1}=(a-1)/2$, $b_{i2}=(a+1)/2$ or vice versa. Thus
\[
a^2-b_{21}^2-\cdots-b_{32}^2\le a^2-2\left({(a-1)^2\over 4}
+{(a+1)^2\over 4}\right)=-1.
\]
Again, since $k\ge 1$, the only possibility is $k=1$, $b_{11}=b_{12}=0$,
and hence $a=1$, giving the following possible choices for $\beta$:
\[
L-E_{21}-E_{31}, \quad L-E_{21}-E_{32}, \quad L-E_{22}-E_{31}, \quad
L-E_{22}-E_{32}.
\]
Note that these four classes, along with $E_{11}$, $E_{12}$,  and
cases (1) and (2) in the $a$ even case, represent the $8$ $(-1)$-curves
in $Y$ which meet $D_1$ transversally, i.e., the curves $L_{1j}$.
Each of these curve classes is
then represented by a unique $\AA^1$-curve, and $N^{\beta}_{v_1}=1$
in these cases.

In the case $a=k=2$, we consider
the curve class $\beta=2L-E_{21}-\cdots-E_{32}\sim D_2+D_3$. Note
that the linear system $|D_2+D_3|$ induces a conic bundle 
$g:Y\rightarrow \PP^1$, and $D_1$ is a $2$-section of $g$, i.e.,
$g|_{D_1}:D_1\rightarrow \PP^1$ is a double cover, necessarily
branched over two points $p_1, p_2\in\PP^1$. Thus the conics
$f^{-1}(p_1)$, $f^{-1}(p_2)$ are also $\AA^1$-curves, now with contact
order $2$ with $D_1$. So $N^{\beta}_{2v_1}=2$.\footnote{In general,
in Gromov-Witten theory, it is not enough to just count the stable maps,
as there may be a virtual count. However, in all the cases just
considered, the stable map $f:C\rightarrow Y$ in question is a closed immersion,
and hence has no automorphisms as a stable map. Further, the obstruction
space to the moduli space of stable maps at the point $[f]$ is
$H^1(C,f^*T_Y(-\log D))$, which is seen without much difficulty to
vanish. Hence each curve in fact contributes $1$ to the Gromov-Witten
number.}

Unfortunately, these are not the only $\AA^1$-curves, as there may be
stable maps $f:(C,p)\rightarrow Y$ which are either not generically
injective or don't have irreducible image. Indeed, one may
have multiple covers of one of the above $\AA^1$-curves already
considered, provided the cover is totally ramified at the point $p$
of contact with $D$. However, for general choice of $(Y,D)$, we will now
show that there is no possibility of reducible images. 

As argued
in \cite{GP}, Lemma 4.2, the image of any $\AA^1$-curve must be a 
union of irreducible
curves each of which intersect the boundary at the same point.
In particular, if $f(C)=C_1\cup\cdots\cup C_n$ is the irreducible 
decomposition, then $D\cap C_1\cap \cdots \cap C_n$ consists of one
point, necessarily contained in $D_1$.

However, since $C_1,\ldots,C_n$ must be a subset of the $10$ curves
identified above, the possibilities are as follows. The first is that
two of these
curves are lines on the cubic surface, and hence we must have three
lines (including $D_1$) intersecting in a common point. Such a 
point on a cubic surface is called an \emph{Eckardt point}, see
\cite{Dolg}, \S9.1.4. However, the set of cubic surfaces containing
Eckardt points is codimension one in the moduli of all cubic surfaces.
Since we may assume $(Y,D)$ is general in moduli (as the Gromov-Witten
invariants being calculated are deformation invariant), we may
thus assume $Y$ has no Eckardt points, so this doesn't occur.

On the other hand, one or both of $C_1,C_2$ could be fibres of
the conic bundle induced by $|D_2+D_3|$. Since two distinct fibres
are disjoint, they can't both be fibres of the conic bundle. 
Further, any line $E$
of the cubic surface intersecting
$D_1$ at one point has $E\cdot (D_2+D_3)=0$, and hence is contained
in a fibre of the conic bundle $g$, and thus is again disjoint
from a different fibre of $g$. 

We thus come to the conclusion that any stable map contributing to the
$\AA^1$-curve count must have irreducible image, and hence be a 
multiple cover of one of the curves discussed above. The moduli space of
such multiple covers is always positive dimensional, but happily the
virtual count has been calculated in \cite{GPS}, Proposition 6.1.
Degree $d$ covers of a non-singular rational curve which meets $D$ 
transversally contributes $(-1)^{d-1}/d^2$, whilst degree $d$ covers
of a non-singular rational curve which is simply tangent to $D$ is
$1/d^2$. Note 
\[
\exp\left(\sum_{d\ge 1}d\cdot {(-1)^{d-1}\over d^2} z^{d\beta} X_1^{-d}\right)
=1+z^{\beta} X_1^{-1}
\]
and
\[
\exp\left(\sum_{d\ge 1}2d\cdot {1\over d^2} z^{d\beta} X_1^{-2d}\right)
={1\over (1-z^{\beta} X_1^{-2})^2}.
\]
From this the result follows.
\end{proof}

We now observe that the cubic surface carries sufficient symmetry
so that the above computation determines the scattering diagram completely.

Noting that the group $\SL_2(\ZZ)$ acts on $\RR^2$
and $-\id$ lies in the centre of $\SL_2(\ZZ)$, we obtain an action of
$\PGL_2(\ZZ)$ on $B=\RR^2/\langle -\id\rangle$. Of course, this action
acts transitively on all the rays of rational slope in $B$, so if we can
show that this action preserves the scattering diagram in a certain sense,
we will have completely determined the scattering diagram.

We first observe that there is a rotational symmetry. For example,
the calculation of $f_{\rho_1}$ equally applies to $f_{\rho_2}$ and
$f_{\rho_3}$, subject to a change of relevant curve classes.
More generally, if we know $f_{\fod}$ for $\fod=\RR_{\ge 0} (av_i+bv_{i+1})$,
then we know it for $S(\fod)$, where $S(av_i+bv_{i+1})=
av_{i+1} + bv_{i+2}$, with indices taken modulo $3$. Here $S$ is
an automorphism of $B$ which lifts to an automorphism of the cover
$\RR^2$, with $S(1,0)=(0,1)$ and $S(0,1)=(-1,1)$ (so that $S(-1,1)=(-1,0)$,
completing the rotation). Thus on the cover, $S$ is represented
by $\begin{pmatrix} 0&-1\\ 1&1\end{pmatrix}$. We also have an action $S^*$
on $H_2(Y,\ZZ)$ given by $S^*(L)=L$, $S^*(E_{ij})=E_{i+1,j}$.
Then we can 
write $f_{S(\fod)}=S^*(f_{\fod})$, where the action of $S^*$ on $f_{\fod}$
is given by $X_i^{-ka}X_{i+1}^{-kb}\mapsto X_{i+1}^{-ka}X_{i+2}^{-kb}$
and $z^{\beta}\mapsto z^{S^*(\beta)}$.  

The second symmetry arises from a birational change to the boundary.
We may blow-up the point of intersection of $D_1$ and $D_2$,
and blow-down $D_3$, to obtain a surface $(Y',D')$ with $Y'\setminus
D'=Y\setminus D$. We use the convention that
$D'=D_1'+D_2'+D_3'$ with $D_1'$ the strict transform of $D_1$, $D_2'$
the exceptional curve of the blow-up, and $D_3'$ the strict transform of
$D_2$.

But in fact $Y'$ is still a cubic surface, and
hence we may apply the calculation of Proposition \ref{prop:first ray}
with respect to the new divisor $D_2'$. Because of the way 
$\AA^1$-curve counts are defined, these counts do not depend on
toric blow-ups and blow-downs of the boundary. Thus if we know a ray
in the scattering diagram for $(Y',D')$, we have a corresponding ray
in the scattering diagram for $(Y,D)$. For example, it is not difficult
to check that for $1\le j\le 8$, the curve in the pencil $|D_3+L_{3j}|$
passing through $D_1\cap D_2$ has strict transform in $(Y',D')$
a line meeting $D_2'$. On the other hand, the strict transform of a
curve of class $D_1+D_2+2D_3$ on $(Y,D)$ which is cuspidal at
$D_1\cap D_2$ is a conic on $Y'$ which meets $D_2'$ tangentially.

To see this as an action on $B$,
let $B'$ be the integral linear manifold with singularities
corresponding to $(Y',D')$. Then there is a canonical piecewise
linear identification of $B'$ with $B$ arising from the description of 
the tropicalization of Remark \ref{rem:trop}.
In particular, this identification sends $v_1'$ to $v_1$,
$v_2'$ with $v_1+v_2$, and $v_3'$ with $v_2$. Thus if we know
a ray $(\fod',f_{\fod'})$ for $(Y',D')$, we obtain a ray
$(\fod,f_{\fod})$ for $(Y,D)$ under this identification. Instead, we can view
this identification as giving an automorphism of $B$, i.e., 
consider the automorphism $T$ given by $v_1\mapsto v_1$, $v_2\mapsto v_1+v_2$
and $v_3\mapsto v_2$. Note this is induced by
$\begin{pmatrix} 1&1\\ 0&1\end{pmatrix}\in\SL_2(\ZZ)$.

It is not difficult to work out the action\footnote{In fact this action
is not unique: it can always be composed with an automorphism
of $H^2(Y,\ZZ)$ preserving the intersection form, permuting the $(-1)$-curves,
and keeping the boundary divisors
$D_1,D_2,D_3$ fixed. We give one possible action.}

 $T^*$ on $H_2(Y,\ZZ)$. It is
\begin{align*}
L\mapsto {} & 2L-E_{31}-E_{32}\\
E_{1j}\mapsto {} &E_{1j}\\
E_{2j}\mapsto {} &L-E_{3j}\\
E_{3j}\mapsto {} &E_{2j}
\end{align*}
for $j=1,2$. Then the symmetry $T$ takes a ray $(\fod,f_{\fod})$ to a ray
$(T(\fod),T^*(f_{\fod}))$, where $T^*(f_{\fod})$ does the obvious thing.
In particular, one ray in $\foD_{\can}$ is
$\fod=\RR_{\ge 0}(v_1+v_2)$ with
\begin{equation}
\label{eq:v1v2ray}
f_{\fod}:={\prod_{j=1}^8 (1+z^{D_3+L_{3j}}X_1^{-1}X_2^{-1})\over
(1-z^{D_1+D_2+2D_3}X_1^{-2}X_2^{-2})^4}
\end{equation}
Happily, this is
the only additional ray we will need to understand other than 
$\rho_1,\rho_2$ and $\rho_3$.

Since $S$ and $T$ generate $\SL_2(\ZZ)$, we have now proved:

\begin{theorem}
\label{thm:scat}
Let $\fod=\RR_{\ge 0} (av_i+bv_{i+1})$ for $a,b\in\ZZ_{\ge 0}$ relatively prime.
Then there exists curve classes $\beta_1,\ldots,\beta_9\in H_2(Y,\ZZ)$
such that
\[
f_{\fod}={\prod_{j=1}^8 (1+z^{\beta_j} X_i^{-a}X_{i+1}^{-b})
\over (1-z^{\beta_9}X_i^{-2a}X_{i+1}^{-2b})^4}.
\]
\end{theorem}

We note that this $\SL_2(\ZZ)$-action has a beautiful explanation in
terms of work of Cantat and Loray \cite{CL}. They describe the $\SL_2(\CC)$
character variety of the four-punctured sphere $S^2_4=S^2\setminus
\{p_1,\ldots,p_4\}$, i.e., the variety
of $\SL_2(\CC)$ representations of the fundamental group $\pi_1(S^2_4)$,
up to conjugation by elements of $\SL_2(\CC)$. 
This character variety is naturally embedded in $\AA^7$ with
coordinates $x,y,z,A,B,C,D$ and has equation
\[
xyz+x^2+y^2+z^2=Ax+By+CZ+D,
\]
i.e., is a family of affine cubic surfaces whose natural compactifications
in $\PP^3$ are then precisely of the form we are considering.

Now $S^2$ can be viewed as a quotient of a torus $T^2=\RR^2/\ZZ^2$ 
by negation, $S^2=T^2/\langle -\id\rangle$, and the map
$T^2\rightarrow S^2$ has four branch points, the two-torsion points
of $T^2$. We take the image of these branch points to be $p_1,\ldots,p_4$,
so that any element of $\SL_2(\ZZ)$ acting on $T^2$ then induces
an automorphism of $S^2_4$, possibly permuting the punctures. Thus
we obtain a $\PGL_2(\ZZ)$ action on $S^2_4$, and hence a $\PGL_2(\ZZ)$
action on the character variety, which in fact is compatible with the
projection to $\AA^4$ with coordinates $A,B,C,D$. An element
of $\PGL_2(\ZZ)$ permutes fibres if it permutes two-torsion points.
Thus, we obtain an action of $\PGL_2(\ZZ)$ on the ``relative''
tropicalization of this family of log Calabi-Yau manifolds, i.e., the set of 
valuations with centers surjecting onto $\AA^4$ and with simple poles
of the relative holomorphic $2$-form. This can be shown to be
the same action considered above generated by $S$ and $T$.
We omit the details.

\section{Broken lines, theta functions and the derivation of the equation}

We now explain how to construct theta functions, and what is special
about the canonical scattering diagram. We first recall the notion
of broken line, fixing here a scattering diagram $\foD$ and a co-Artinian
ideal $I\subseteq\kk[P]$.

\begin{definition}
A \emph{broken line} $\gamma$ in $(B,\Sigma)$ for $q\in B_0(\ZZ)$
and endpoint $Q\in B_0$ is a proper continuous piecewise integral
affine map $\gamma:(-\infty,0]\rightarrow B_0$, real numbers
$t_0=-\infty< t_1<\cdots <t_n=0$, and monomials $m_i$, $1\le i\le n$,
satisfying the following properties:
\begin{enumerate}
\item $\gamma(0)=Q$.
\item $\gamma|_{[t_{i-1},t_i]}$ is affine linear for all $i$, and 
$\gamma([t_{i-1},t_i])$ is contained
in some two-dimensional cone $\sigma_{j,j+1}\in\Sigma$, where $j$
depends on $i$. Further,
$m_i=c_iX_j^aX_{j+1}^b$ for some $a,b\in\ZZ$, $a,b$ not both zero,
$c_i\in \kk[P]/I$, and $\gamma'(t)=-av_j-bv_{j+1}$ for any $t\in
(t_{i-1},t_i)$.
\item If $q\in \Int(\sigma_{j,j+1})$ for some $j$, we can write 
$q=av_j+bv_{j+1}$ for some $a,b\in \ZZ_{\ge 0}$, and then
$\gamma((-\infty,t_1])\subset \sigma_{j,j+1}$ and
$m_1=X_j^aX_{j+1}^b$. If $q\in \rho_j$ for some $j$, then
$\gamma((-\infty,t_1])$ is either contained in $\sigma_{j-1,j}$ or
$\sigma_{j,j+1}$, and writing $q=av_j$, we have $m_1=X_j^a$.
\item If $\gamma(t_i)$ lies in the interior of a maximal cone of
$\Sigma$ then $\gamma(t_i)$ lies in the support of a ray $(\fod,f_{\fod})$
and $\gamma$ passes from one side of $\fod$ to the other, so that
$\theta_{\gamma,\fod}$ is defined.
Then $m_{i+1}$ is a monomial in $\theta_{\gamma,\fod}(m_i)$. In other
words, we expand the expression $\theta_{\gamma,\fod}(m_i)$ into
a sum of monomials, and choose $m_{i+1}$ to be one of the terms of this
sum.
\item If $\gamma(t_i)\in\rho_j$ for some $j$, $\gamma$ passes
from $\sigma_{j-1,j}$ to $\sigma_{j,j+1}$, and
$m_i=c_iX_{j-1}^aX_j^b$, then
$m_{i+1}$ is a monomial in the expression
\[
c_i(z^{[D_j]}X_j^{-D_j^2}f_{\rho_j}X_{j+1}^{-1})^{a} X_j^{b}.
\]
If, on the other hand, $\gamma$ passes from $\sigma_{j,j+1}$ to
$\sigma_{j-1,j}$, then, with $m_i=c_iX_j^aX_{j+1}^b$, $m_{i+1}$ is
a monomial in the expression
\[
c_iX_j^a(z^{[D_j]}X_j^{-D_j^2}f_{\rho_j}X_{j-1}^{-1})^b.
\]
In other words, in the first case, the monomial $m_i$, written
in the variables $X_{j-1}, X_j$, is rewritten, using the defining
equation of the ring $R_{j,I}$, in the variables $X_j, X_{j+1}$.
The second case is similar.
\end{enumerate}
\end{definition}

\begin{definition}
Let $q\in B_0(\ZZ)$ and
$Q\in B_0$ be a point with irrational coordinates. Then we define
\[
\vartheta_{q,Q}=\sum_{\gamma} \Mono(\gamma)
\]
where the sum is over all broken lines for $q$ with endpoint $Q$,
and $\Mono(\gamma)$ denotes the last monomial attached to $\gamma$.

We extend this definition to $q=0\in B(\ZZ)\setminus B_0(\ZZ)$ by
setting 
\[
\vartheta_{0,Q}=\vartheta_0=1.
\]
\end{definition}

It follows from the definition of broken line that if $Q\in\sigma_{i,i+1}$
then $\vartheta_{q,Q}\in R_{i,i+1,I}$. 

\begin{definition}
We say $\foD$ is \emph{consistent} if for all $q\in B_0(\ZZ)$ and co-Artinian
ideals $I$,
\begin{enumerate}
\item If $Q,Q'\in \sigma_{i,i+1}$ are points with irrational coordinates
and $\gamma$ is a path in $\sigma_{i,i+1}$ joining $Q$ to $Q'$, then
\[
\theta_{\gamma,\foD}(\vartheta_{q,Q})= \vartheta_{q,Q'}.
\]
\item
If $Q\in\sigma_{i-1,i}$, 
$Q'\in \sigma_{i,i+1}$ are chosen sufficiently close to $\rho_i$ 
such that there is no non-trivial ray
of $\foD_I$ between $Q$ and $\rho_i$ or between $Q'$ and $\rho_i$,
then there exists an element $\vartheta_{q,\rho_i}\in R_{i,I}$ whose images
in $R_{i-1,i,I}$ and $R_{i,i+1,I}$ are $\vartheta_{q,Q}$ and $\vartheta_{q,Q'}$
respectively.
\end{enumerate}
\end{definition}

One of the main theorems of \cite{GHK11}, namely Theorem 3.8, states that
$\foD_{\can}$ is a consistent scattering diagram.

The benefit of a consistent scattering diagram is that the $\vartheta_{q,Q}$
for various $Q$ can then be glued to give a global function
$\vartheta_q\in \Gamma(X_{I,\foD}^{\circ},\shO_{X_{I,\foD}^{\circ}})$.
This allows us to construct a partial compactification $X_{I,\foD}$ of
$X^{\circ}_{I,\foD}$ by setting
\[
X_{I,\foD}:=\Spec\Gamma(X_{I,\foD}^{\circ},\shO_{X_{I,\foD}^{\circ}}),
\]
and the existence of the theta functions $\vartheta_q$ guarantees that
this produces a flat deformation of $\VV_n$ over $\Spec A_I$,
see \cite{GHK11}, \S2.3 for details.

Morally, another way to think about this is that we are embedding 
$X_{I,\foD}^{\circ}$ in $\AA^n_{A_I}$ using the theta functions
$\vartheta_{v_1},\ldots,\vartheta_{v_n}$, and then taking the closure.

\begin{example}
Unfortunately, in the cubic surface example, it is very difficult to
write down expressions for theta functions. While for any fixed ideal
$I$, $\theta_{q,Q}$ is a finite sum of monomials, in fact if we take
the limit over all $I$ we obtain an infinite sum. Here are
some very simple examples of this. Take $Q=\alpha v_1+\beta v_2$
for some irrational $\alpha,\beta\in \RR_{>0}$, and take $q=v_1$. We give
examples of broken lines for $q$ ending at $Q$. Consider
a ray $\fod=\RR_{\ge 0}(a v_1+b v_2)$ with $(a-1)/b<\alpha/\beta<a/b$. Note
that given the choice of $Q$, there are an infinite number of
choices of relatively prime $a,b$ satisfying this condition.

We will construct a broken line as depicted in Figure \ref{Figure2}. The
monomial attached to the segment coming in from infinity is $X_1$.
If we bend along the ray $\fod$, we apply $\theta_{\gamma,\fod}$ to
$X_1$.  By Theorem \ref{thm:scat}, $f_{\fod}$ contains a non-zero term 
$cX_1^{-a}X_2^{-b}$ for some $c\in\kk[P]$,
so $\theta_{\gamma,\fod}(X_1)$ contains a term $c'X_1^{1-a}X_2^{-b}$.
Choosing this monomial, we now proceed in the direction $(a-1,b)$.
In particular, take the bending point to be
\[
Q':=\left({a\over b}((1-a)\beta+b\alpha), (1-a)\beta+b\alpha\right),
\]
which lies on $\fod$ because $(1-a)\beta+b\alpha>0$ by the assumption
that $(a-1)/b<\alpha/\beta$. Then 
\[
Q-Q'=({a\over b}\beta-\alpha)(a-1,b).
\]
Thus the broken line reaches $Q$, as depicted.
So we have indeed constructed a broken line, and there are 
an infinite number of such broken lines (albeit only a finite number
modulo any co-Artinian ideal $I$). 
\begin{figure}
\input{Figure2.pstex_t}
\caption{}
\label{Figure2}
\end{figure}

Of course, here we are using only one term from the infinite power
series expansion of $f_{\fod}$ and only considering one possible
bend, and we already have an infinite number of broken lines. 
We believe it would be extremely difficult to get a useful description
of all broken lines, and hence broken lines provide a useful theoretical,
but not practical, description of theta functions.
\end{example}

This shows that if we take the limit over all $I$ and obtain a formal scheme
$\widehat{\foX}\rightarrow 
\Spf \widehat{\kk[P]}$, there is no hope to express the theta
functions as algebraic expressions. Thus it is perhaps a bit of a surprise
that often the relations satisfied by these theta functions are much
simpler, so that we can extend the construction over $\Spec \kk[P]$.
In the case of the cubic surface, we will see this explicitly by
using the product rule for theta functions in Theorem 2.34 of \cite{GHK11}:

\begin{theorem}
\label{multrule}
Let $p_1,p_2 \in B(\bZ)$. In the
canonical expansion
\[
\vartheta_{p_1} \cdot \vartheta_{p_2} =
\sum_{r \in B(\bZ)} \alpha_{p_1p_2r} \vartheta_r,
\]
where $\alpha_{p_1p_2r}\in \kk[P]/I$ for each $q$, we have
\[
\alpha_{p_1p_2r} = \sum_{\gamma_1,\gamma_2} c(\gamma_1)c(\gamma_2),
\]
where the sum and notation is as follows. We fix
$z\in B_0$ a point very close to $r$ contained in the interior of
a cone $\sigma_{i,i+1}$ for some $i$. We then sum over all broken
lines $\gamma_1$, $\gamma_2$ for $p_1$, $p_2$ satisfying: (1) Both 
broken lines have endpoint $z$. (2)
If $\Mono(\gamma_j)=c(\gamma_j)X_i^{a_j}X_{i+1}^{b_j}$ with
$c(\gamma_j)\in \kk[P]/I$, $j=1,2$, then
$r=(a_1+a_2)v_i+(b_1+b_2)v_{i+1}$.
\end{theorem}

We shall see that in the case of the cubic surface, only a very small
part of $\foD_{\can}$ is necessary to find the equation of the mirror.

There is one more ingredient for the calculation of the
equation, namely the notion of a \emph{min-convex function} in the context
of a scattering diagram. Let $F:B\rightarrow\RR$ be a piecewise
linear function on $B$. If $\gamma$ is a broken line, we obtain a
(generally discontinuous) function on $(-\infty,0]$, the domain of $\gamma$,
written as $t\mapsto dF(\gamma'(t))$. This means that at a time $t$,
provided $F$ is linear at $\gamma(t)$, we evaluate the differential
$dF$ at $\gamma(t)$ on the tangent vector $\gamma'(t)$. Thus
$dF(\gamma'(\cdot))$ is a piecewise constant function.

We say $F$ is \emph{min-convex} if for any broken line $\gamma$,
$dF(\gamma'(\cdot))$ is a decreasing function: see \cite{GHKK},
Definition 8.2, where the definition is given in a slightly different
context. The use of such a function is that \cite{GHKK}, Lemma 8.4
applies, so that $F$ is \emph{decreasing} in the sense of
\cite{GHKK}, Definition 8.3., i.e., if the coefficient $\alpha_{p_1p_2r}
\not=0$, then 
\begin{equation}
\label{eq:decreasing}
F(r)\ge F(p_1)+F(p_2).
\end{equation}
Indeed, suppose that $\gamma_1$, $\gamma_2$ are broken lines
for $p_1,p_2$ respectively contributing to the expression for
$\alpha_{p_1p_2r}$. Note that $F(p_i)=-dF(\gamma_i'(t))$ for $t\ll 0$, while
$r=-\gamma_1'(0)-\gamma'_2(0)$. Thus for $t\ll 0$, 
\begin{equation}
\label{eq:precise decreasing}
F(r)-F(p_1)-F(p_2)=\sum_{i=1}^2(dF(\gamma_i'(t))-dF(\gamma_i'(0))),
\end{equation}
which is positive under the decreasing assumption.
 
We note \cite{B18} also makes use of such a function (with the opposite sign
convention). Barrott, however, used a computer program to enumerate
all contributions to the products, as his main goal was to find the
mirror to a degree $2$ del Pezzo, which has a considerably more complex
equation than the mirror to the cubic. In the case of the cubic surface,
the products can be computed by hand.

In our case, we may take $F=\langle K_Y,\cdot\rangle$.
Note this
pulls back to the PL function on the cover $\RR^2\rightarrow B=\RR^2/\langle
-\id\rangle$ which corresponds to $K_{\widetilde Y}$.

It is easy to check that in fact $dF(\gamma'(\cdot))$ decreases whenever a broken
line crosses one of the rays $\rho_i$ (this is just local convexity
of $F$) or when a broken line bends, and hence $F$ is decreasing.
However, it will be important to quantify by how much $dF(\gamma'(\cdot))$
changes with each of these occurences. For example, suppose a broken line
passes from $\sigma_{1,2}$ into $\sigma_{2,3}$ without bending, with
tangent direction $a v_1+bv_2$, necessarily with $a<0$. Then via parallel
transport of this tangent vector into $\sigma_{2,3}$, we can rewrite 
the vector using the relation $v_1+v_3=-D_2^2 v_2$, i.e.,
$a v_1 + b v_2$ is rewritten as $(a+b) v_2 -a v_3$. Thus
$dF(\gamma'(\cdot))$ takes the value $-(a+b)$ before crossing $\rho_2$,
and the value $-(a+b)+a$ after crossing $\rho_2$, hence decreasing as
$a<0$.

If $\gamma$ bends in, say, $\sigma_{i,i+1}$, then $\gamma'$ changes
by some $av_i+bv_{i+1}\not=0$ for $a,b\ge 0$. But $dF(av_i+bv_{i+1})=
-a-b$, so $dF(\gamma'(\cdot))$ changes by $-(a+b)<0$.

Note that if $\gamma$ bends when it crosses $\rho_i$, in fact
$dF(\gamma'(t))$ decreases by at least $2$. These observations will
be crucial for bounding the search for possible broken lines contributing
to the product.

We now calculate the key products necessary to prove the main theorem.

\begin{lemma} 
\label{lem:products}
We have the following products:
\begin{align*}
\vartheta_{v_i}^2= {} & \vartheta_{2v_i}+2z^{D_j+D_k},\quad \{i,j,k\}=\{1,2,3\}
\\
\vartheta_{v_1}\vartheta_{v_2}= {} &\vartheta_{v_1+v_2}+z^{D_3}\vartheta_{v_3}
+\sum_{j=1}^8 z^{D_3+L_{3j}}\\
\vartheta_{v_1+v_2}\vartheta_{v_3} = {} &
z^{D_1}\vartheta_{2v_1}+z^{D_2}\vartheta_{2v_2}+
\vartheta_{v_1}\sum_j z^{D_1+L_{1j}}+\vartheta_{v_2} \sum_j z^{D_2+L_{2j}}
+\sum_\pi z^{\pi^*H}+8z^{D_1+D_2+D_3}.
\end{align*}
\end{lemma}

\begin{proof}
We consider first $\vartheta_{v_i}^2$. By symmetry, we can take
$i=1$. Since $F(v_1)=-1$, if
$\vartheta_r$ contributes to this product, we must have
$-2\le F(r) \le 0$, the first inequality from \eqref{eq:decreasing}
and the second since $F$ is non-positive. Let $\gamma_1$, $\gamma_2$ be
broken lines contributing to $\alpha_{v_1v_1r}$ as in Theorem
\ref{multrule}.
It follows immediately from
\eqref{eq:precise decreasing} that if $F(r)=-2$, then
$\gamma_i$ neither bends nor crosses a wall. It is then obvious
the only possible $r$ in this case is $r=2v_1$. 
Fixing $z\in \Int(\sigma_{1,2})$
near $2v_1$, we obtain the contribution from two broken lines as in the left
in Figure \ref{figurev12}: this is responsible for the $\vartheta_{2v_1}$
term.

There are only three points $r\in B(\ZZ)$ with $F(r)=-1$, 
namely $r=v_i$, $i=1,2,3$.
Now if a pair of broken lines $\gamma_1$, $\gamma_2$ contributes to 
$\alpha_{v_1v_1r}$, then one of the $\gamma_j$ either bends or
crosses one of the $\rho_k$. Such a possibility can be ruled out, however.
It is easiest to work on the cover $\RR^2\rightarrow B=\RR^2/
\langle-\id\rangle$,
bearing in mind that there are two possible initial directions for the
lifting of a
broken line $\gamma_i$, namely it can come in parallel to $\RR_{\ge 0}(1,0)$
or parallel to $\RR_{\ge 0}(-1,0)$.
If $r=v_2$ or $v_3$, we can fix $z$ in the interior of $\sigma_{2,3}$,
and then both broken lines must cross rays to reach $z$. If $r=v_1$, we
may take $z$ in the interior of $\sigma_{1,2}$, and then the only possibility
is that one of the $\gamma_i$ bends. However, if $\gamma_i$ bends 
at any ray of $\foD_{\can}$ not supported on one of the $\rho_i$, then
$dF$ decreases by at least $2$, ruling out this possibility. Thus
we can rule out the case $F(r)=-1$. We shall omit this kind of analysis in the 
sequel, as it is straightforward.

Finally, if $F(r)=0$, then $r=0$. Taking $z$ in $\sigma_{1,2}$ close to
the origin, we obtain the possibility shown on the right-hand side
of Figure \ref{figurev12}. This actually represents two possibilities,
as the labels $\gamma_1$ and $\gamma_2$ can be interchanged. Each such pair
of broken lines contributes $z^{D_2+D_3}\vartheta_0$, recalling that
$\vartheta_0=1$. One checks easily that there are no possibilities
where one of the broken lines bends. This gives the claimed description
of $\vartheta_{v_1}^2$.

\begin{figure}
\input{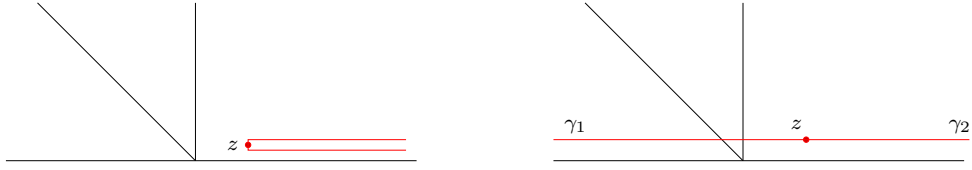}
\caption{The contributions to the product $\vartheta_{v_1}^2$.
In the left-hand picture, the two broken lines in fact lie on top of
each other, but we depict them as distinct lines with endpoint $z$.}
\label{figurev12}
\end{figure}

\medskip

Turning to $\vartheta_{v_1}\cdot\vartheta_{v_2}$, if $F(r)=-2$, then
again broken lines can't bend or cross walls. In this case, the only
possibility is as depicted on the left in Figure \ref{figurev1v2},
contributing the term $\vartheta_{v_1+v_2}$.

If $F(r)=-1$, then $r=v_1,v_2$ or $v_3$. By putting the endpoint
$z$ in $\sigma_{1,2}, \sigma_{2,3}$ or $\sigma_{2,3}$ respectively,
a quick analysis shows the only possible contribution is from the right-hand
picture in Figure \ref{figurev1v2}, contributing $z^{D_3}\vartheta_{v_3}$.

\begin{figure}
\input{figurev1v2.pstex_t}
\caption{}
\label{figurev1v2}
\end{figure}

\begin{figure}
\input{figurev1v2two.pstex_t}
\caption{}
\label{figurev1v2two}
\end{figure}

Finally, if $F(r)=0$, again $r=0$. Taking $z$ near $\rho_1$ and the origin 
in the
interior of $\sigma_{1,2}$, we now obtain the possibility of $\gamma_2$
bending along the ray $\fod=\RR_{\ge 0}(1,1)$ as depicted in Figure
\ref{figurev1v2two}. The bend on $\gamma_2$ is calcluated by seeing
how $\theta_{\fod,\gamma_2}$ acts on the initial monomial $X_2$,
i.e., $X_2\mapsto X_2f_{\fod}$. By the form given for $f_{\fod}$ in 
\eqref{eq:v1v2ray}, we get the given expression for $\vartheta_{v_1}
\cdot \vartheta_{v_2}$.

\medskip

Turning to $\vartheta_{v_1+v_2}\cdot \vartheta_{v_3}$, this time
we have the possible range $-3\le F(r)\le 0$. However, $F(r)=-3$
is impossible, as this does not allow either $\gamma_i$ to cross a
ray $\rho_j$, and necessarily $\gamma_1$ and $\gamma_2$ come in from
infinity in different cones.

If $F(r)=-2$, as at least one of the $\gamma_i$ crosses a ray $\rho_j$,
no bends are possible. One then sees the two possibilities in
Figure \ref{figurev12v3one}. These give rise to the contributions
$z^{D_2}\vartheta_{2v_2}$ and $z^{D_1}\vartheta_{2v_1}$ respectively.

\begin{figure}
\input{figurev12v3one.pstex_t}
\caption{}
\label{figurev12v3one}
\end{figure}

If $F(r)=-1$, then a bend is also permitted, and $r=v_1,v_2$ or $v_3$.
By placing $z$ in the interiors of $\sigma_{1,2}$, $\sigma_{1,2}$
or $\sigma_{2,3}$ respectively, near $\rho_1$, $\rho_2$, or $\rho_3$,
one rules out $v_3$ as a possibility and has as remaining possibilities
as in Figure \ref{figurev12v3two}. These contribute
$\vartheta_{v_1}\sum_j z^{D_1+L_{1j}}$ and $\vartheta_{v_2}\sum_j 
z^{D_2+L_{2j}}$ respectively.

\begin{figure}
\input{figurev12v3two.pstex_t}
\caption{}
\label{figurev12v3two}
\end{figure}

Finally we have $F(r)=0$, i.e., $r=0$. We put $z$ in the interior
of $\sigma_{1,2}$
near the origin, close to $\rho_1$. Then $\gamma_1$ stays in the interior
of $\sigma_{1,2}$, and therefore
can only bend at a ray of $\foD_{\can}$ intersecting the interior
of $\sigma_{1,2}$. However, if it bends at any ray other than
$\fod=\RR_{\ge 0}(1,1)$, $dF$ decreases by at least $3$, while
$\gamma_2$ crosses some $\rho_i$, which would require $F(r)\ge 1$. 
On the other hand, as $\gamma_1$ is initially parallel to $\fod$,
it can't cross $\fod$, and hence $\gamma_1$ doesn't bend. This leaves
only the possibility depicted in Figure \ref{figurev12v3three}.
This involves rewriting $X_3$ using the equation
$X_2X_3=f_{\rho_1} z^{[D_1]} X_1$, i.e., $X_3=f_{\rho_1}z^{[D_1]}X_1X_2^{-1}$
and choosing a monomial of the form $cX_1^{-1}X_2^{-1}$ from this
expression. We thus need to consider the coefficient of $X_1^{-2}$
in $f_{\rho_1}$, and this is $\sum_{1\le j< j'\le 8} z^{L_{1j}+L_{1j'}}
+4 z^{D_2+D_3}$. Thus we get a contribution of
\begin{equation}
\label{eq:first version}
\sum_{1\le j< j' \le 8} z^{D_1+L_{1j}+L_{1j'}}+4z^{D_1+D_2+D_3}
\end{equation}
to the product. We can give a clearer description of this expression,
however. Consider the class $L_{1j}+L_{1j'}$. If $L_{1j}\cap L_{1j'}
=\emptyset$, then $L_{1j}$ and $L_{1j'}$ can be simultaneously contracted.
One can easily check that given this choice of $j, j'$, there are unique
pairs $L_{2k}, L_{2k'}$ and $L_{3\ell}, L_{3\ell'}$ such that all
six of these curves can be simultaneously contracted to give a morphism
$\pi:(Y,D)\rightarrow (Y',D')$, where $Y'\cong \PP^2$ and $D'$ is the
image of $D$. This morphism is in fact induced by the two-dimensional
linear system $|D_1+L_{1j}+L_{1j'}|$. In particular, 
$D_1+L_{1j}+L_{1j'}=\pi^*H$ where $H$ is the class of a line on $Y'$. 

On the other hand, a plane in $\PP^3$ containing both $D_1$ and $L_{1j}$
contains a third line $L_{1j'}$ for some $j'$ with $L_{1j}\cap L_{1j'}$
a point. Thus the set $\{L_{1j}\}$ is partitioned into four pairs,
with $L_{1j}$, $L_{1j'}$ in the same pair if $L_{1j}\cap L_{1j'}\not=\emptyset$,
in which case $D_1+L_{1j}+L_{1j'}\sim D_1+D_2+D_3$. Thus we can express
\eqref{eq:first version} as
\[
\sum_\pi z^{\pi^*H}+8z^{D_1+D_2+D_3}.
\]
This is responsible for the last contribution to $\vartheta_{v_1+v_2}
\vartheta_{v_3}$.
\end{proof}

\begin{figure}
\input{figurev12v3three.pstex_t}
\caption{}
\label{figurev12v3three}
\end{figure}

\begin{proof}[Proof of Theorem \ref{thm:main theorem}.]
Using the lemma, we calculate
\begin{align*}
\vartheta_{v_1}\vartheta_{v_2}\vartheta_{v_3}= {} &
\left(\vartheta_{v_1+v_2}+z^{D_3}\vartheta_{v_3}+\sum_{j=1}^8 z^{D_3+L_{3j}}
\right) \vartheta_{v_3}\\
= {} &z^{D_1}\vartheta_{2v_1}+z^{D_2}\vartheta_{2v_2}+z^{D_3}\vartheta_{v_3}^2
+\sum_i\left(\sum_jz^{L_{ij}}\right)z^{D_i}\vartheta_{v_i}
+\sum_{\pi} z^{\pi^*H} + 8z^{D_1+D_2+D_3}\\
= {} &z^{D_1}\vartheta_{v_1}^2+z^{D_2}\vartheta_{v_2}^2+z^{D_3}\vartheta_{v_3}^2
+\sum_i\left(\sum_jz^{L_{ij}}\right)z^{D_i}\vartheta_{v_i}
+\sum_{\pi} z^{\pi^*H} + 4z^{D_1+D_2+D_3},
\end{align*}
as desired.
\end{proof}

\begin{remark}
We have constructed a family of cubic surfaces over $S=\Spec\kk[P]$
where $P=\NE(Y)$, where $Y$ is a non-singular cubic surface. 
There is an intriguing slice of this family
related to the Cayley cubic which has already made its appearance
in Remark \ref{rem:cayley cubic}.

We may obtain the Cayley cubic as follows. Take four general lines
$L_1,\ldots,L_4$ in $\PP^2$, giving $6$ pairwise intersection points.
By blowing up these six points, we obtain a surface
$Y$. Note that the cone of effective curves of $Y$ is different
than that of a general cubic surface because it contains some $(-2)$-curves.
However, for this discussion it is convenient to keep $P$ to be the
cone of effective curves of a general cubic surface.

The strict transforms
of the four lines become disjoint $(-2)$-curves which may be contracted,
giving a cubic surface $Y'$ with four ordinary double points: this is
the Cayley cubic.

If we take $\overline D_1$ to be the line joining $L_1\cap L_2$ and
$L_3\cap L_4$, $\overline D_2$ the line joining $L_1\cap L_3$ and
$L_2\cap L_4$, and $\overline D_3$ the line joining 
$L_1\cap L_4$ and $L_2\cap L_3$, and let $D_i$ be the strict transform
of $\overline D_i$ in $Y$, we obtain a log Calabi-Yau pair
$(Y,D=D_1+D_2+D_3)$ as usual. With suitable labelling of the exceptional
curves, we can write the classes $F_i$ of the strict transforms of the
$L_i$ as
\begin{align*}
F_1:= {} & L-E_{11}-E_{21}-E_{31}\\
F_2:= {} & L-E_{11}-E_{22}-E_{32}\\
F_3:= {} & L-E_{12}-E_{21}-E_{32}\\
F_4:= {} & L-E_{12}-E_{22}-E_{31}
\end{align*}

Now consider the big torus $S^{\circ}=\Spec\kk[P^{\gp}]\subset
S=\Spec\kk[P]$, and consider further the subscheme 
$T\subseteq S^{\circ}$ defined by the equations $z^{F_i}=1$, $i=1,\ldots,4$.
Then $T=\Spec \kk[P^{\gp}/{\bf F}]$,
where ${\bf F}$ is the subgroup of $P^{\gp}$
generated by the curve classes $F_i$. However, it is not difficult to see
that this quotient has two-torsion. Indeed, 
\[
2E_{11}-2E_{12}=F_3+F_4-F_1-F_2\in {\bf F},
\]
while $E_{11}-E_{12}\not\in {\bf F}$. We also have
\[
E_{11}-E_{12}\equiv E_{21}-E_{22}\equiv E_{31}-E_{32} \mod
{\bf F}.
\]
In fact, $T$ has two connected components, 
one containing the identity
element in the torus $S^{\circ}$, and the other satisfying the equations
$z^{E_{i1}}=-z^{E_{i2}}$ for $i=1,2,3$. Let $T'$ denote this latter
component, and restrict the family of cubic surfaces given over $S$
in Theorem \ref{thm:main theorem} to $T'$. One may check that
the equation becomes
\begin{equation}
\label{eq:Cayley eq}
\vartheta_1\vartheta_2\vartheta_3=\sum_i z^{D_i}\vartheta_i^2-4z^{D_1+D_2+D_3}
\end{equation}
and that this is in fact a family of Cayley cubics.

Somewhat more directly, one may also consider the restriction of the
scattering diagram to $T'$. For example, consider the ray $(\rho_1,
f_{\rho_1})$ described in Proposition \ref{prop:first ray}.
We note that the set of eight lines ${\bf L}:=\{L_{1i}\}$ split into two groups
of four, 
\begin{align*}
{\bf L}_1:= {} &\{E_{11}, L-E_{21}-E_{31}, L-E_{22}-E_{32}, 2L-E_{11}
-E_{21}-E_{22}-E_{31}-E_{32}\}\\
{\bf L}_2:= {} & 
\{E_{12}, L-E_{21}-E_{32}, L-E_{22}-E_{31}, 2L-E_{12}
-E_{21}-E_{22}-E_{31}-E_{32}\}
\end{align*}
such that if $L_1,L_2\in {\bf L}$ lie in the same ${\bf L}_i$ then
$L_1-L_2\in {\bf F}$, so
$z^{L_1}=z^{L_2}$ on $T'$.
On the other hand, if they do not lie in the same
${\bf L}_i$, then $2(L_1-L_2)\in {\bf F}$ and
$z^{L_1}=-z^{L_2}$ on $T'$. Further, if $L\in {\bf L}$, then $2L-D_2-D_3
\in {\bf F}$.

Thus we obtain, after restriction to $T'$,
that, for any choice of $L\in {\bf L}$,
\[
f_{\rho_1}= {[(1+z^LX_1^{-1})(1-z^LX_1^{-1})]^4\over 
(1-z^{2L}X_1^{-2})^4} = 1.
\]
So the ray becomes trivial after restriction to $T'$. Following the
argument of \S2, one then sees all rays of $\foD_{\can}$ become
trivial after restriction to $T'$. Thus the equation \eqref{eq:Cayley eq}
may in fact be obtained
using broken lines as carried out in this section, 
but this time all broken lines involved are straight.

It is unusual that the trivial scattering diagram is consistent
when the affine manifold $B$ has a singularity. In the K3 case, 
there is a similar situation arising when $B$ is an affine two-sphere
arising as a quotient of $\RR^2/\ZZ^2$ via negation.

It is also intriguing that the relevant subtorus $T'\subseteq S^{\circ}$
is translated, i.e., does not pass through the origin; 
we speculate that this might have an explanation in terms
of orbifold Gromov-Witten invariants of the Cayley cubic.
\end{remark}

\section{The enumerative interpretation of the equation}

There is also a much more recent interpretation of the multiplication
law of Theorem \ref{multrule}, which gives a Gromov-Witten interpretation
for the $\alpha_{p_1p_2r}\in \kk[P]/I$, writing
\[
\alpha_{p_1p_2r}=\sum_{\beta\in P} N^{\beta}_{p_1p_2r} z^{\beta}
\]
with $N^{\beta}_{p_1p_2r}\in\QQ$.
In \cite{GS18}, Gross and Siebert explain how to associate certain
Gromov-Witten numbers to the data $\beta$, $p_1,p_2$ and $r$.
The details are given in \cite{GS19}. The paper
\cite{KY19} gives a different approach for the same ideas.
In general, the construction
of these invariants is quite subtle, but happily in the case at hand,
all invariants will be easy to calculate. Roughly put, the
numbers are defined as follows. 

The choice of $r$ defines a
choice of stratum $Z_r\subseteq Y$. Indeed, if $\sigma\in\Sigma$ is
the minimal cone containing $r$, then $\sigma$ corresponds to
a stratum of $Y$, i.e., if $r=0$, $Z_r=Y$, if $r\in\Int(\rho_i)$
then $Z_r=D_i$, and if $r\in\Int(\sigma_{i,i+1})$ then $Z_r=D_i\cap D_{i+1}$.
We choose a general point $z\in Z_r$. 

Then $N^{\beta}_{p_1p_2r}$
is a count of the number of stable logarithmic maps
$f:(C,x_1,x_2,x_{\out})\rightarrow Y$ such that:
\begin{enumerate}
\item  the order of tangency of $f$ at $x_j$ with $D_k$ is 
$\langle D_k, p_j\rangle$, $j=1,2$.
\item the order of tangency of $f$ at $x_{\out}$ with $D_k$
is $-\langle D_k,r\rangle$ and $f(x_{\out})=z$.
\end{enumerate}
Note this involves negative orders of tangency at $x_{\out}$, and
defining this is subtle. See \cite{GS18} and \cite{ACGS18} for more
details for how this notion is defined. Here, usually $r=0$, so we
will only have a couple of cases where we have to accept the possibility
of a negative order of tangency. For the complete technically correct
definition of the above invariant, see \cite{GS19}, \S3.

\cite{GS18}, Proposition 2.4 is quite useful in telling us when the
relevant moduli space is empty. In particular, that proposition
tells us that  $N^{\beta}_{p_1p_2r}\not=0$ implies that if $D'$ is
any divisor supported on $D$, then
\begin{equation}
\label{eq:balancing}
\beta\cdot D' = \langle D',p_1\rangle+\langle D',p_2\rangle - \langle D',
r\rangle.
\end{equation}
In particular, if we take $D'=D=-K_Y$, we obtain
\[
\langle D,p_1\rangle+\langle D,p_2\rangle=\beta\cdot D+\langle D,r\rangle.
\]
As $D$ is ample in the case of the
cubic surface, in particular $\beta\cdot D\ge 0$ for any effective curve class,
so we get the stronger result that if $N^{\beta}_{p_1p_2r}\not=0$, then
\[
\langle D,p_1\rangle+\langle D,p_2\rangle\ge\langle D,r\rangle\ge 0,
\]
These formulae may be compared with \eqref{eq:decreasing} and 
\eqref{eq:precise decreasing}.  In fact, $F=\langle K_Y,\cdot\rangle$, 
and the above formulae play the same role as those of \S3.

We now revisit the calculation of Lemma \ref{lem:products}. The arguments
which follow are necessarily sketchy as we have not given a full definition
of the invariants here. We trust the arguments should be sufficiently
plausible, however.

For example, let us reconsider the product $\vartheta_{v_1}^2$.
We see that if $N^{\beta}_{v_1v_1r}\not=0$ then
$\langle D,r \rangle\le 2$ with equality if and only if
$\beta=0$. Thus if we do have equality, then any map $f:(C,x_1,x_2,x_{\out})
\rightarrow Y$ contributing to $N^{\beta}_{v_1v_1r}$ is constant.
This is discussed in \cite{GS19}, Lemma 1.15, where
it is shown that if $N^0_{p_1p_2r}\not=0$, then $p_1,p_2$ lie in the same
cone and $r=p_1+p_2$. Further, $N^0_{p_1p_2r}=1$ in this case.
In particular, $N^0_{v_1v_1,2v_1}=1$. This gives the
contribution $\vartheta_{2v_1}$ to $\vartheta_{v_1}^2$.

If $\beta\cdot D=1$, then the only possibilities for $r$ are $v_1,v_2$ or
$v_3$. Suppose $r=v_1$. As $\beta\cdot D=1$, $\beta$ is the class of a 
line on $Y$. Since we choose $z\in Z_r$ general, none of
the lines $L_{ij}$ pass through $z$ and thus the image of $f$ may not
be $L_{ij}$. If the image of $f$ is $D_2$, then $f$ has non-trivial
contact with $D_3$, which is not allowed. Similarly, the image of $f$
may not be $D_3$. Finally, if the image of $f$ is $D_1$, 
\eqref{eq:balancing} yields a contradiction if one takes
$D'=D_1$. Thus we eliminate this case. The cases that $r=v_2,v_3$
are similarly ruled out.

Finally, we have one remaining case, when $\beta\cdot D=2$
and $r=0$. Thus we consider conics which meet $D_1$ transversally
at two points (labelled $x_1$, $x_2$), are disjoint from $D_2$ and
$D_3$, and have a third point $x_{\out}$ which coincides with a fixed
general point $z\in Y$. It is easy to see that any such conic must
be in the linear system $|D_2+D_3|$, and there is one such conic
passing through $z$. However, as the labels of the intersection points
of the conic with $D_1$ can be interchanged, in fact 
$N^{D_2+D_3}_{v_1v_20}=2$. This gives the second term in the
product $\vartheta_{v_1}^2$.

\medskip

We now move onto $\vartheta_{v_1}\cdot\vartheta_{v_2}$. A similar analysis
with the possible degree of the class $\beta$ leads to the following
choices. First, we may have $\beta=0$, and so $r=v_1+v_2$ and
$N^0_{v_1,v_2,v_1+v_2}=1$, giving the first contribution to the product.

Next, if $\beta\cdot D=1$, then $r=v_1, v_2$ or $v_3$. As before,
$\beta$ must be the class of a line, and as before, we must have
$\beta=D_i$ for some $i$ as otherwise the image of $f$ will not contain
$z$. If $\beta=D_3$, we can identify $C$ with $D_3$,
taking $x_1$ to be the intersection of $D_1$ and $D_3$ and
$x_2$ to be the intersection of $D_2$ with $D_3$. Since $\beta\cdot D_3=-1$,
\eqref{eq:balancing} tells us that $r=v_3$.
After fixing $z\in D_3$,
we take $x_{\out}=z$. One can show that $N^{D_3}_{v_1v_2v_3}=1$.\footnote{
We note that the full verification of this statement is somewhat
involved, as one must construct the unique punctured curve in the
relevant moduli space and show that it is unobstructed. However, this
is fairly routine for those familiar with log Gromov-Witten theory,
and we omit the details here as it would involve introducing a lot
of additional technology into this survey.}
This contributes
the term $z^{D_3}\vartheta_{v_3}$ to the product. On the other hand,
if $\beta=D_1$, taking $D'=D_1$ in \eqref{eq:balancing}
results in a contradiction regardless
of the choice of $r=v_k$, and the same holds if $\beta=D_2$. Thus
there are no further choices.

Finally, if $\beta\cdot D=2$, $r=0$, we fix $z\in Y$ general.
We now need to consider conics which meet both $D_1$ and $D_2$
transversally, pass through $z$, and are disjoint from $D_3$. 
There are a total of
$27$ conic bundles on $Y$: for $E$ the class of a line on $Y$, $|D-E|$
is a pencil of conics. Thus one easily checks that only eight of 
these have the correct intersection properties with $D$, 
precisely conics
of classes $D_3+L_{3j}$, $1\le j\le 8$. 
For each $j$, there is precisely one conic in the pencil
$|D_3+L_{3j}|$ passing through $z$. This is responsible for the last
term in the product $\vartheta_{v_1}\cdot\vartheta_{v_2}$.

\medskip

We now turn to the product $\vartheta_{v_1+v_2}\cdot \vartheta_{v_3}$.
As $v_1+v_2$ and $v_3$ do not lie in a common cone of $\Sigma$,
constant maps cannot occur. Thus we are faced with the 
possibilities $1\le \beta\cdot D\le 3$.

If $\beta\cdot D=1$, the same arguments as before reduce to
the possibilities that $\beta=D_1$, $D_2$ or $D_3$. 
First $\beta=D_3$ is impossible: any curve with contact order
at a point given by $v_1+v_2$ must pass through $D_1\cap D_2$.
However, in each of the other cases,
there is exactly one allowable map. For example, in case
$\beta=D_1$, we take $z\in D_1$ general, identify $C$ with $D_1$,
take $x_1$ to be the intersection point of $D_1$ and $D_2$,
$x_2$ the intersection point of $D_1$ and $D_3$, take $x_{\out}=z$,
and take $r=2v_1$.
Again it is possible to show that these curves exist as punctured
logarithmic curves, and $N^{D_i}_{v_1+v_2,v_3,2v_i}=1$ for $i=1,2$.
This gives the first two terms in the product.

If $\beta\cdot D=2$, then $r=v_i$ for some $i$,
and we must consider conics which pass through
$D_1\cap D_2$, are transversal to $D_3$, and pass through an additional
point $z\in D_i$. 
We may now show the image of any punctured map contributing
to $N^{\beta}_{v_1+v_2,v_3,v_i}$ is reducible. If the image is an
irreducible conic, that conic must pass through $D_1\cap D_2$, intersect
$D_3$ in at least one point, and pass through the generally
chosen point $z\in D_i$. This implies that $\beta\cdot (D_1+D_2+D_3)\ge 3$.
Since $D_1+D_2+D_3$ is the class of a hyperplane
section of the cubic surface, this contradicts $\beta$ being a degree $2$
class. If, on the other hand, the image of the punctured map is a line
(hence the punctured map is a double cover), this line must be
$D_1$ or $D_2$, being the only lines passing through $D_1\cap D_2$.
Thus $\beta = 2D_1$ or $2D_2$. However, this case is
ruled out via an application of \eqref{eq:balancing}. 

Thus necessarily the image of the punctured map is a union of two
lines. The only lines passing through $D_1\cap D_2$
are $D_1$ and $D_2$, and thus $\beta = D_i + L$ for $i=1$ or $2$ and
$L$ some other line. As the image of
$f$ must be connected, this only leaves the option of $\beta=D_1+L_{1j}$,
$\beta=D_2+L_{2j}$, or $\beta=D_j+D_k$. The third case
can be ruled out from \eqref{eq:balancing}, and for the first
two cases, one can show that $N^{\beta}_{v_1+v_2,v_3, v_i}
=1$. This gives the third and fourth terms
in the expression for $\vartheta_{v_1+v_2}\cdot\vartheta_{v_3}$.

Finally, we consider the case of $\beta\cdot D=3$, so that $\beta$
is a cubic. There are two choices. Either $\beta$ is the class
of a twisted cubic, i.e., $-2=\beta\cdot (\beta+K_Y)$, or $\beta$
is the class of an elliptic curve, i.e., $0=\beta\cdot (\beta+K_Y)$.
Now if $\beta$ is the class of a twisted cubic, it is easy to see
that the linear system $|\beta|$ is two-dimensional and induces
a morphism $\pi:Y\rightarrow Y'\cong \PP^2$. If in addition,
$\beta\cdot D_i=1$ for each $i$ (which follows from \eqref{eq:balancing}), 
$\pi$ maps
$D_1,D_2$ and $D_3$ to lines in $Y'$. Hence there is a one-to-one
correspondence between such classes $\beta$ and morphisms
$\pi:Y\rightarrow Y'$ as before.

Given such a morphism, $\beta=\pi^*H$, and there is a unique twisted
cubic in the linear system $|\beta|$ passing through both $z$
and $D_1\cap D_2$. Thus $N^{\beta}_{v_1+v_2, v_3,0}=1$. This gives
the fifth term in the expression for $\vartheta_{v_1+v_2}\cdot
\vartheta_{v_3}$.

Finally, if $\beta$ is the class of an elliptic curve of degree $3$,
it is necessarily planar, and hence $\beta=D$. We now
calculate $N^{\beta}_{v_1+v_2,v_3,0}$. First, there is a pencil
of plane cubics passing through $D_1\cap D_2$ and $z$. If
$\ell\subseteq\PP^3$ denotes the line joining these points, then each 
element of the pencil is of the form
$H\cap Y$ for $H\subseteq \PP^3$ a plane containing $\ell$.
To study this pencil, we may blow-up its basepoints, which are the
three points of $\ell \cap Y$. This gives a rational elliptic surface
$g:\tilde Y\rightarrow \PP^1$. Via a standard
Euler characteristic computation, such a surface is expected to have $12$
singular fibres. However, note that if $H$ contains $D_i$, $i=1$
or $2$, then $H\cap Y$ is a union $D_i\cup C$ of a line and a conic. 
In general, $C$ intersects $D_i$ in two points. By normalizing one
of these two nodes, we obtain a stable map to $Y$. However, none of
these maps can be equipped with the structure of a stable log map
because the point of normalization on the conic maps into $D$ and
has non-zero contact order with $D$, yet it is not a marked point.

Since we have just seen that two of the fibres of this elliptic fibration
are of Kodaira type $I_2$, 
this leaves $8$ additional nodal elliptic curves. By normalizing
the node, one obtains a genus zero stable map with the desired
intersection behaviour with $D$. This yields the last term
in the description of $\vartheta_{v_1+v_2}\cdot \vartheta_{v_3}$.

\bigskip

We close by noting that the Frobenius structure conjecture
(see the first arXiv version of \cite{GHK11}, Conjecture 0.9, or \cite{M19}
and \cite{KY19}) 
gives us another explanation for the constant term (i.e., coefficient
of $\vartheta_0$,)
$\sum_{\pi} z^{\pi^*H} + 10 z^{D_1+D_2+D_3}$
in the equation defining the mirror to the cubic surface.
Here we write $10$ rather than $4$ as we rewrite the equation
for the mirror in terms of $\vartheta_{2v_i}$ instead of
$\vartheta_{v_i}^2$.

Indeed, the Frobenius conjecture implies
that we may calculate the constant term in the triple product
$\vartheta_{v_1}\vartheta_{v_2}\vartheta_{v_3}$ as
$\sum_{\beta} N^{\beta}_{v_1v_2v_30}z^{\beta}$ where, roughly,
$N^{\beta}_{v_1v_2v_30}$ is a count defined as follows. 
Fix $z\in Y$ general and
$\lambda\in\overline{\shM}_{0,4}$.
Then we count four-pointed
stable log maps $f:(C,x_1,x_2,x_3,x_{\out})\rightarrow Y$ such that
$f$ meets $D_i$ transversally at $x_i$, $f(x_{\out})=z$,
and the modulus of the stabilization of $C$
is $\lambda$. This can be viewed as fixing the cross-ratio of the four points
$x_1,x_2,x_3,x_{\out}$ to be $\lambda$. This part of the Frobenius conjecture
is shown in \cite{GS19} and \cite{KY19}, and see also \cite{M19} for
related results.

The class $\beta$ of such a curve $C$ must satisfy $\beta\cdot D=3$,
so $\beta$ is either a twisted cubic or a plane cubic. In the former
case, one immediately recovers $\sum_{\pi} z^{\pi^*H}$. Indeed,
if one fixes $z\in \PP^2$ and a cross-ratio $\lambda$,
there is a unique line $H$ in $\PP^2$ passing through $z$ 
such that the cross-ratio of $z$ and
the three points of intersection of $H$ with the boundary divisor is
$\lambda$. 

The count of plane cubics is more subtle. In this case, it is easiest
to fix the modulus of the stabilization of $C$ by insisting the
stabilization is a singular curve, with $x_2,x_3$ on one irreducible
component and $x_1,x_{\out}$ on the other. 
There are the following possibilities.
\begin{enumerate}
\item The image of $f$ is a union of three lines. This cannot occur,
as such a curve does not pass through a general $z\in Y$.
\item The image of $f$ is the union of a line and a conic, $E\cup Q$. 
Suppose $E\not=D_i$ for any $i$. Then $E$ meets $D$ at one point and is
rigid, hence does not pass through $z$. Thus three of the four 
marked points of $C$ must lie in $Q$. This contradicts the choice of
modulus. Thus $E=D_i$ for some $i$, and $Q\in |D-D_i|$. In particular,
as $Q$ is irreducible, $Q$ is disjoint from $D_j, D_k$ for $\{i,j,k\}=
\{1,2,3\}$. Thus $D_i$ must contain those marked points mapping to
$D_j$ and $D_k$, so necessarily $D_i=D_1$. In particular, $C$ is
the normalization of $D_1\cup Q$ at one of the two nodes, and the marked
point $x_1$ is the point of $Q$ mapping to the chosen node.
Note that this marking is what allows us to count this curve, as opposed
to the same curve considered in the contribution to the constant term
in $\vartheta_{v_1+v_2}\cdot \vartheta_{v_3}$. Because of the choice
of nodes, this gives two curves of class $D$.
\item The image of $f$ is an irreducible nodal cubic. In order for the
domain to have the given modulus, $x_2$ and $x_3$ must lie on a contracted
component of $C$, i.e., $C=C_1\cup C_2$ with $x_2,x_3\in C_1$, $x_1,
x_{\out}\in C_2$, $f|_{C_1}$ constant with image $D_2\cap D_3$,
and $f(C_2)$ a nodal cubic. The count is now exactly the same as
in the case of the contribution of nodal cubics to $\vartheta_{v_1+v_2}
\cdot \vartheta_{v_3}$, and we have $8$ such nodal cubics.
\end{enumerate}
This explains the term $10z^{D_1+D_2+D_3}$.


\begin{thebibliography}{99}

\bibitem[AC14]{AC14} D.~Abramovich, Q.~Chen: \emph{Stable logarithmic maps to
Deligne-Faltings pairs II,} Asian J. Math. {\bf 18} (2014), 465--488.

\bibitem[ACGS19]{ACGS18} D.~Abramovich, Q.~Chen, M.~Gross and B.~Siebert,
\emph{Punctured Gromov-Witten invariants}, preprint, 2019, available at
\url{http://www.dpmms.cam.ac.uk/~mg475/punctured.pdf}.

\bibitem[AMW]{AMW}  
D.~Abramovich, S.~Marcus, and J.~Wise, 
\emph{Comparison theorems for Gromov–Witten invariants of smooth pairs 
and of degenerations,}
Annales de l'Institut Fourier, {\bf 64} (2014) 1611-1667.

\bibitem[AW18]{AW18} D.~Abramovich and J.~Wise: \emph{Birational invariance in 
logarithmic Gromov-Witten theory,} Compos.\ Math.\ {\bf 154} (2018),  595--620. 

\bibitem[Ar19]{Ar} H.~Arg\"uz: \emph{Canonical scattering for Log lalabi--Yau
surfaces}, in preparation.

\bibitem[B18]{B18} L.~Barrott, \emph{Explicit equations for mirror families
to log Calabi-Yau surfaces,} preprint, 2018.

\bibitem[CL]{CL} S.~Cantat, F.~Loray:
\emph{Dynamics on character varieties and Malgrange irreducibility of 
Painlev\'e VI equation,}
Ann. Inst. Fourier (Grenoble) {\bf 59} (2009), 2927--2978. 

\bibitem[CPS]{CPS} M.~Carl, M.~Pumperla, and B.~Siebert,
A tropical view of Landau-Ginzburg models, available at
http://www.math.uni-hamburg.de/home/siebert/preprints/LGtrop.pdf

\bibitem[C14]{Chen14} Q.~Chen: \emph{Stable logarithmic maps to Deligne-Faltings
pairs I,} Ann. of Math. (2) {\bf 180} (2014), 455--521.

\bibitem[Do]{Dolg} I.~Dolgachev, \emph{Classical algebraic geometry:
A modern view}, Cambridge University Press, 2012.

\bibitem[GHK11]{GHK11} M.~Gross, P.~Hacking and S.~Keel, \emph{Mirror
symmetry for log Calabi-Yau surfaces I}, Publ.\ Math.\
Inst.\ Hautes ́\'Etudes Sci., {\bf 122}, (2015) 65–168.

\bibitem[GHKII]{GHKII} M.\ Gross, P.\ Hacking, S.\ Keel:
        \emph{Mirror symmetry for log Calabi-Yau surfaces II},
        draft.

\bibitem[GHK13]{GHK13} M.~Gross, P.~Hacking, and S.~Keel, \emph{Birational 
geometry of cluster algebras}, Algebr. Geom.~2 (2015), no.~2, 137--175.

\bibitem[GHKK]{GHKK} M.\ Gross, P.\ Hacking, S.\ Keel, M.\ Kontsevich:
        \emph{Canonical bases for cluster algebras}, Journal of the AMS,
        {\bf 31}, (2018), 497--608.

\bibitem[GP]{GP} M.\ Gross, R.\ Pandharipande: 
\emph{Quivers, curves, and the tropical vertex,} Port. Math. {\bf 67} (2010), 
211--259. 

\bibitem[GPS]{GPS} M.\ Gross, R.\ Pandharipande, B.\ Siebert,
        \emph{The tropical vertex},
        Duke Math.\ J.\ {\bf 153}, (2010) 297--362.

\bibitem[GS11]{GS11} M.~Gross, B.~Siebert, From real affine geometry to
complex geometry, Annals of Mathematics, {\bf 174}, (2011), 1301-1428.

\bibitem[GS18]{GS18} M.\ Gross, B.\ Siebert,
\emph{Intrinsic mirror symmetry and punctured Gromov-Witten invariants.}
Algebraic geometry: Salt Lake City 2015, 199--230, Proc.\ Sympos.\ Pure Math., 
{\bf 97.2}, Amer.\ Math.\ Soc., Providence, RI, 2018. 

\bibitem[GS19]{GS19} M.\ Gross, B.\ Siebert,
\emph{Intrinsic mirror symmetry}, preprint, 2019.

\bibitem[GS11]{JAMS} M.~Gross and B.~Siebert: \emph{Logarithmic
Gromov-Witten invariants},
J.\ Amer.\ Math. Soc., {\bf 26} (2013). 451--510.

\bibitem[KY19]{KY19} S.~Keel and T.~Yu: \emph{The Frobenius structure Conjecture
for affine log CYs containing a torus,} preprint, 2019.

\bibitem[M19]{M19} T.~Mandel: \emph{Theta bases and log Gromov-Witten 
invariants of cluster varieties}, preprint, 2019.

\bibitem[Reid]{Reid} M.~Reid: 
\emph{Undergraduate algebraic geometry.} London Mathematical Society Student 
Texts, 12. Cambridge University Press, Cambridge, 1988. {\rm viii}+129 pp. 

\bibitem[Ob04]{Ob04} A.~Oblomkov, Double affine Hecke algebras
of rank $1$ and affine cubic surfaces, Int. Math. Res. Not.~2004, no. 18,
877--912.

\end{thebibliography}
\end{document}